\documentclass[12pt]{amsart}

\usepackage{amsmath} 
\usepackage{hyperref} 
\usepackage{amsfonts}
\usepackage{amssymb}
\usepackage{amsthm}
\usepackage{latexsym}
\usepackage{color}
\usepackage{enumerate}
\usepackage{mathrsfs}
\usepackage{setspace}
\usepackage{tikz}
\usetikzlibrary{matrix}

\newtheorem{thm}{Theorem}[section]

\newtheorem{lemma}[thm]{Lemma}

\newtheorem{coro}[thm]{Corollary}
\newtheorem{conjecture}[thm]{Conjecture}

\newtheorem{remark}[thm]{Remark}

\newtheorem{definition}[thm]{Definition}

\newtheorem{thm*}{Theorem}[section]
\numberwithin{equation}{section}

\newcommand{\pr}{\partial}
\newcommand{\veps}{\varepsilon}

\def\dv{\textnormal{div}}
\def\M{\mathscr{M}}

\def\grad{\nabla}

\def\R{\mathbb{R}}
\def\area{\textmd{area}}

\def\vol{\mathrm{vol}}
\def\R{\mathbb{R}}

\def\S{\Sigma}
\def\({\left(}
\def\){\right)}
\def\a{\alpha}

\def\={\stackrel{(n=2)}{=}}
\def\To{\longrightarrow}
\def\graph{\textnormal{graph}}

\def\D{\textnormal{D}}

\def\[{\,[\hspace{-.9ex}[\,}
\def\]{\,]\hspace{-.9ex}]\,}
\def\bM{\mathbf{M}}

\def\m{\mathbf{m}}
\def\V{\mathcal{V}}

\newcommand{\be}{\begin{equation}}
\newcommand{\ee}{\end{equation}}
\newcommand{\bee}{\begin{equation*}}
\newcommand{\eee}{\end{equation*}}

\DeclareMathOperator{\diam}{diam}
\DeclareMathOperator{\minA}{minA}
\newcommand{\GHto}{\stackrel { \textrm{GH}}{\longrightarrow} }

\newcommand{\Lipto}{\stackrel { \textrm{Lip}}{\longrightarrow} }
\newcommand{\rstr}{\:\mbox{\rule{0.1ex}{1.2ex}\rule{1.1ex}{0.1ex}}\:}
\newcommand{\mass}{{\mathbf M}}
\newcommand{\inj}{\textrm{inj}}
\newcommand{\bdry}{{\partial}}
\newcommand{\set}{\textrm{set}}
\def\G{\mathcal{G}}
\def\Mda{\mathcal M^{D_0}_{A_0}}
\def\M{\mathcal{M}}
\def\Md{\mathcal M^{D_0}}
\newcommand{\intcurr}{{\mathbf I}} 
\newcommand{\spt}{\operatorname{spt}}
\newcommand{\Lip}{\operatorname{Lip}}

\begin{document}

\title[Stability of graphical tori]{Stability of graphical tori with almost nonnegative scalar curvature}

\author[{Cabrera Pacheco}]{Armando J. {Cabrera Pacheco}}\thanks{AJCP is grateful to the Carl Zeiss Foundation for its generous support.}
\address[Armando J. {Cabrera Pacheco}]{Department of Mathematics, Universit\"at T\"ubingen, T\"ubingen, Germany.}
\email{cabrera@math.uni-tuebingen.de}
\author[Ketterer]{Christian Ketterer}\thanks{CK is funded by the Deutsche Forschungsgemeinschaft (DFG, German Research Foundation) -- Projektnummer 396662902, ``Synthetische Kr\"ummungsschranken durch Methoden des Optimal Transports''.}
\address[Christian Ketterer]{Department of Mathematics, University of Toronto, Canada.}
\email{ckettere@math.toronto.edu}
\author[Perales]{Raquel Perales}
\address[Raquel Perales]{Instituto de Matem\'aticas, Universidad Nacional Aut\'onoma de M\'exico, Oaxaca, Mexico}
\email{raquel.peralesaguilar@gmail.com}

\thanks{The authors were partially supported by NSF DMS-1309360 and NSF DMS-1612049. }

\maketitle

\begin{abstract}
By works of Schoen--Yau and Gromov--Lawson any Riemannian manifold with nonnegative scalar curvature and diffeomorphic to a torus is isometric to a flat torus. Gromov conjectured subconvergence of tori with respect to a weak Sobolev type metric when the scalar curvature goes to $0$. We prove flat and intrinsic flat subconvergence to a flat torus for noncollapsing sequences of $3$-dimensional tori $M_j$ that can be realized as graphs of certain functions defined over flat tori satisfying a uniform upper diameter bound and scalar curvature bounds of the form $R_{g_{M_j}} \geq -1/j$. We also show that the volume of the manifolds of the convergent subsequence converges to the volume of the limit space.  We do so adapting results of Huang--Lee, Huang--Lee--Sormani and  Allen--Perales--Sormani.  Furthermore, our results also hold when the condition on the scalar curvature of a torus $(M, g_M)$ is replaced by a bound on the quantity $-\int_T  \min\{R_{g_M},0\} d{\vol_{g_T}}$, where $M=\graph(f)$, $f: T \to \R$ and $(T,g_T)$ is a flat torus. Using arguments developed by Alaee, McCormick and the first named author after this work was completed, our results hold for dimensions $n \geq 4$ as well. 
\end{abstract}


\section{Introduction}\label{sec-intro}

The celebrated scalar torus rigidity theorem says that any Riemannian manifold that is diffeomorphic to an $n$ dimensional torus and has nonnegative scalar curvature must be isometric to a flat torus. This rigidity statement follows from the fact that an $n$-torus cannot carry a metric of positive scalar curvature. The results were proven for $n \leq 7$ using minimal surfaces theory by Schoen and Yau  \cite{Schoen-Yau-min-surf,Schoen-Yau-pos-scalar}, and by Gromov and Lawson using the Lichnerowicz formula for spin manifolds \cite{Gromov-Lawson-1980} for $n \geq 8$.  In \cite{Gromov-Dirac} Gromov addressed the corresponding stability problem.
\begin{conjecture}[Gromov, Section 5.4 \cite{Gromov-Dirac}]\label{conjGro}
"There is a particular 'Sobolev type weak metric' in the space of $n$-manifolds $X$, such that, for example, tori $(X,g)$ with $R_g\geq -\epsilon$, when properly normalized, (sub)converge to flat tori for $\epsilon\rightarrow 0$, but these $X$ may, in general, diverge in stronger metrics."
\end{conjecture}

We note that Gromov showed that for any manifold $M$ endowed with $C^2$ Riemannian metrics $g_j$ satisfying $R_{g_j} \geq k$ which converge to a $C^2$ Riemannian metric $g_\infty$ in the local $C^0$ sense, then $R_{g_\infty} \geq k$, where $R_{g}$ denotes the scalar curvature of a Riemannian metric $g$ and $k: M \to \R$ is a  continuous function \cite{Gromov-Dirac}.  Then Bamler obtained this result using Ricci flow \cite{Bamler-16}.  On the other hand, Basilio and Sormani constructed sequences of $3$-dimensional tori with almost nonnegative scalar curvature with either no Gromov--Hausdorff limit or a non-smooth Gromov--Hausdorff limit \cite{Basilio-Sormani-1}. These examples have increasingly thin wells with positive scalar curvature, that disappear under intrinsic flat convergence, surrounded by an annular region that satisfies $R_{g_j} \ge -\frac{1}{j}$. In \cite{Sormani-scalar} Sormani made the following refined conjecture. 
\begin{conjecture}[Sormani \cite{Sormani-scalar}]\label{conjecture}
Let $M_j$ be a sequence of Riemannian manifolds diffeomorphic to a $3$-torus such that 
\begin{align*}
\vol(M_j) = V_0,\,\, \diam(M_j) \le D_0 \,\, \text{and} \,\, \minA(M_j) \ge A_0 > 0,
\end{align*}
where $\minA(M_j)=\inf\{\vol(\Sigma): \Sigma \mbox{ is a closed minimal surface in} M_j\}$. If the scalar curvature of $M_j$,  $R(M_j)$, satisfies $R(M_j) \ge -\frac{1}{j}$ then there is a subsequence $M_{j_k}$ converging in intrinsic flat sense to a flat torus $T$ and possibly $\vol(M_{j_k}) \to \vol(T)$. 
\end{conjecture}

The flat distance $d_F$  between integral currents is a classical notion from Geometric Measure Theory introduced by Federer--Fleming. The intrinsic flat distance, $d_{\mathcal{F}}$, 
was introduced by Sormani and Wenger \cite{SorWen2} applying work of 
Ambrosio and Kirchheim \cite{AK}.  The intrinsic flat distance between two compact oriented Riemannian manifolds of the same dimension $M_i$ endowed with their canonical currents $[M_i]$ is defined as the infimum of  the flat distances  $d_F(\varphi_{1\sharp}[M_1], \varphi_{2\sharp}[M_2])$, where $\varphi_i: M_i \to Z$ are distance preserving embeddings into any complete metric space. See Section \ref{sec-back} for the precise definitions. In this way, a sequence of spheres with spikes that contain decreasingly amounts of volume converges to a sphere in intrinsic flat sense \cite{SorWen2}.  

We note that the volume function is lower semicontinuous with respect to the flat and intrinsic flat distances. Hence, the last sentence of Conjecture \ref{conjecture} is meaningful. The equality on the volume bound and the diameter bound in Conjecture \ref{conjecture} prevent collapsing and expanding. There are examples 
of Gromov--Hausdorff and intrinsic flat convergent sequences of Riemannian manifolds $M_j$ satisfying 
$\minA(M_j) \to 0$, in which the scalar curvature blows up to negative infinity  \cite{Basilio-Dodziuk-Sormani}.

We prove Conjecture \ref{conjecture} and Gromov's Conjecture \ref{conjGro}  for a special class of 3 dimensional graphical Riemannian tori. 
Let $M=\graph(f) \subset T\times \R $ with the induced Riemannian metric where $T$ is a 3 dimensional flat torus and $f:T\rightarrow \mathbb{R}$ a $C^4$ function.  Now let $\G$ be the set of functions that satisfy:
\begin{enumerate}
\item $\max_{T} f =(f \circ p) \vert_{\pr D} \equiv 0$, where $p: \R^3 \to T$ is a Riemannian covering map and $D$ the closure of a fundamental domain.
\item For almost every $h\in f(T)$ the level set $\Sigma_h= f^{-1}(h)$ is strictly mean convex with respect to $-\frac{Df}{|Df|}$, i.e., ${\bf H}_{\S_h} \cdot -\frac{D f}{|D f|} >0$, where  ${\bf H}_{\S_h}$ is the mean curvature vector of $\S_h$ in the hyperplane.
\footnote{The mean curvature convention is that spheres have positive mean curvature with respect to the inner pointing normal vector.}
\item For almost every $h\in f(T)$ the level set $\tilde \Sigma_h= (f \circ p)^{-1}(h)$ is outer minimizing.
\end{enumerate}
Here Condition (1) guarantees that each face of the fundamental domain $D$ is a minimal surface in $M$.

We then define $\mathcal{M}$ as the class of graphical tori that arise from all functions in $\mathcal G$.
We let $\Mda$ be the class of manifolds $M \in \M$ that satisfy,
$$
\diam (M)\leq D_0 \,\, \&\,\, \minA(M) \geq A_0.
$$
Finally, we let $\mathcal M^{D_0}$ be the class of manifolds $M \in \M$  that only satisfy
$\diam (M)\leq D_0.$

First we prove subconvergence in the Federer--Fleming flat sense.   Fix a flat torus $T$ and for any $h \in \R$  let $T  \times \{h\}  \subset T\times \R$ have the orientation induced from the orientation of $T \times \R$.

\begin{thm}\label{thm-Flat2}
Let $M_i \in \mathcal M^{D_0}$  such that $M_i$ arises from $f_i: T \to \R$, $i\in \mathbb N$, and $T$ is fixed.  
Assume that 
\begin{align}
R(M_i) \to 0.
\end{align}
Then, there is a subsequence of $\{M_i\}$ denoted in the same way converging in flat sense in $T\times \mathbb{R}$ to $T\times \{0\}$
and
\be
\vol(M_i)  \to  \vol (T). 
\ee
\end{thm}

We remark that Theorem \ref{thm-Flat2} does not imply the following theorem because the graphs $M_i \subset T \times \R$ are not embedded  via distance preserving maps.  Furthermore, in Theorem \ref{thm-IF2} we do not necessarily fix a flat torus. 

\begin{thm}\label{thm-IF2}
Let $M_i \in \Md$ be a sequence of tori, $i \in \mathbb N$, that satisfies
\begin{align}
R(M_i) \to 0.
\end{align}
Then  either $\vol(M_i)  \to 0$ or  there is a subsequence of $M_i$ denoted in the same way converging in intrinsic flat sense to a flat torus $T_\infty$
and
\be 
\vol(M_i) \to \vol(T_\infty).
\ee
If $M_i \in \Mda$ then $\vol(M_i)  \to 0$ does not occur. 
\end{thm}

The definition of $\M$ and the proofs of our flat convergence result follow by adapting the techniques developed by Huang and Lee in \cite{HL} to study the corresponding stability question for the positive mass theorem.  They showed that sequences of $n$-dimensional, $n \geq 3$, graphical asymptotically flat manifolds with nonnegative scalar curvature converge with respect to the flat distance to the Euclidean space whenever the total mass goes to zero.

In \cite{HLS} Huang, Lee and Sormani proved the analogous result with respect to the intrinsic flat distance.  To prove our intrinsic flat convergence result we apply only the volume convergence from Theorem \ref{thm-Flat2} which is an adaptation of Huang--Lee-Sormani's arguments and then apply a recent theorem by Allen, the third named author and Sormani \cite{APS}. We note that the first named author has also adapted Huang and Lee's techniques to prove a stability result for graphical asymptotically hyperbolic manifolds \cite{Cabrera}.

{

Let us describe some of the properties of the classes of manifolds we consider.  The scalar curvature of graphical Riemannian tori can be written in divergence form and 
is therefore linked to the level sets of $f$ by the divergence theorem.
Explicitly, we can derive the following inequality,
\begin{align}\label{uuu}
\m( f) \geq \int_{\S_h}    \frac{|D f|^2}{1+|D f|^2}H_{\S_h} \, d\vol_{\S_h},
\end{align}
where $h$ is any regular value of $f$, $\S_h$ the corresponding level set, $H_{\S_h}$ its mean curvature and $\m(f) = -\int_T R^-_{g_M}d\vol_T$. Since a small negative lower bound on the scalar curvature of $M$ translates into a small upper bound for $\m(f)$, we can study the stability of graph tori with almost nonnegative scalar curvature using $\m(f)$.  
Using (1) and (3) one is able to show that the volume function of the level sets of $f$ is non decreasing.  Using (2)
and \eqref{uuu} we obtain a differential inequality for the volume function of the level sets  that allows us to define a suitable
$h_0\in [\min f,\max f]$ such that we can control the volume of $\Omega_{h_0} = f^{-1}(-\infty, h_0)$ and $M\backslash \Omega_{h_0}$ in terms of $\m(f)$.
These volume estimates allow us to prove flat convergence for a larger class of manifolds and we obtain as corollary Theorem \ref{thm-Flat2}. These estimates also imply the volume convergence statement that appears in Theorem \ref{thm-IF2}.

Conjecture \ref{conjecture} has been proved for tori with singly and doubly warped product Riemannian metrics \cite{Warped}. In this case, uniform subconvergence is obtained and hence Gromov--Hausdorff  and intrinsic flat subconvergence to the same limit space is achieved.  Furthermore, sequences of conformal deformations of a smooth closed Riemannian manifold of dimension $n$ with uniform volume bounds and $L^{n/2}$ bounds on their scalar curvatures have been studied in \cite{AlCaTa}. There, it is shown that such sequences subconverge in Gromov--Hausdorff sense. 

The paper is organized as follows.  In Section \ref{sec-back}  we go over the notions of Gromov--Hausdorff distance, flat distance and intrinsic flat distance.  In Section \ref{sec-GraphTori} we give the definition of graph tori, the negative scalar curvature excess $\m(f)$, derive \eqref{uuu}, fix some notation about fundamental domains and list some of the properties of $\M$.   In Section \ref{sec-proofsF},   we define the aforementioned value $h_0 \in [\min f, \max f]$ which gives control on the volume of 
$\Omega_{h_0}=f^{-1}(-\infty, h_0)$ and $M\setminus \Omega_{h_0}$ in terms of $\m(f)$.   With these estimates we show flat convergence and volume convergence for the sequences of manifolds considered in the theorems above.

The last section is devoted to prove Theorem \ref{thm-IF2}.  We remark that the $\minA$ uniform bound is only used to ensure that the limit torus $T_\infty$ is $3$-dimensional, i.e., the sequence is noncollapsing.  From the previous sections, we do have volume convergence. To prove intrinsic flat convergence we apply a result in \cite{APS} that holds 
for closed oriented Riemannian manifolds of the form $(M,g_j)$, c.f.  Theorem \ref{thm:aps}.  Since in Theorem \ref{thm:aps}  the manifold $M$ is fixed but in our case $M_i \in  \Md$ vary we use a triangle inequality argument.

After this work was completed, a result concerning the stability of a quasi-local positive mass theorem was obtained by Alaee, McCormick and the first named author \cite{ACS}. By using the techniques developed by them, Lemma \ref{lemma-h0-estimate} in Section \ref{sec-proofsF} can be established without the restriction on the dimension, promoting our results to dimensions $n >3$.

\setcounter{tocdepth}{1}
\tableofcontents

\subsection*{Acknowledgments}
The authors would like to thank Prof. Gromov and Prof. Sormani for including the third named author in their Emerging Topics on Scalar Curvature Workshop at IAS where both Gromov and Huang provided key feedback that lead to weakened hypotheses on our main theorems.   We would also like to thank Sormani for funding our earlier travels to workshops in Montreal and to the Institute of Mathematics of the National Autonomous University of Mexico where we completed part of this project (DMS-1309360, DMS-1612049).  We would like to thank the organizers of the Summer School on Geometric Analysis in July 2017 at the Fields Institute in Toronto where we first began working on this problem with Robin Neumayer.  We would like to thank R. Neumayer and R. Haslhofer for their many helpful discussions. The third named author thanks the hospitality of the Scuola Normale Superiore di Pisa where part of this project was written while she was visiting Prof. Luigi Ambrosio.  We would like to thank Prof. Sormani for suggesting this problem. We are very grateful for her continuous encouragement and support. Finally we would like to thank the anonymous referee for giving important comments and remarks that improved the final version of this article.


\section{Background: Notions of Convergence}\label{sec-back}

In this section we review Gromov--Hausdorff distance between metric spaces, flat distance between integral currents on Euclidean space, integral currents on metric spaces and intrinsic flat distance between integral current spaces. We state results concerning these distances  that we will use in subsequent sections.  In each subsection we give references where the material is explained in detail.


\subsection{Gromov--Hausdorff Convergence}

For details regarding the Gromov--Hausdorff distance we refer to \cite{burago}.

\begin{definition}[Gromov]\label{defn-GH} 
The Gromov--Hausdorff distance between two 
compact metric spaces $\left(X, d_X\right)$ and $\left(Y, d_Y\right)$
is defined as
\be \label{eqn-GH-def}
d_{GH}\left(X,Y\right) := \inf  \, d^Z_H\left(\varphi\left(X\right), \psi\left(Y\right)\right),
\ee
where the infimum is taken over all compact metric spaces $(Z,d_Z)$ and distance preserving embeddings $\varphi: X \to Z$ and $\psi:Y\to Z$.  The Hausdorff distance in $Z$ is defined as
\be
d_{H}^Z\left(A,B\right) = \inf\left\{ \epsilon>0: A \subset T_\epsilon\left(B\right) \textrm{ and } B \subset T_\epsilon\left(A\right)\right\},
\ee
where $T_\epsilon (A)$ denotes the $\epsilon$ tubular neighborhood of $A$ in $Z$. 
\end{definition}

This is a distance on the class of compact metric spaces in the sense that $d_{GH}\left(X,Y\right)=0$
if and only if there exists an isometry between $X$ and $Y$ \cite{Gromov-metric}.   

The following embedding theorem holds \cite{Gromov-poly}.
\begin{thm}[Gromov] \label{Gromov-Z}
If a sequence of compact metric spaces, $X_j$,
converges in Gromov-Hausdorff sense to a compact metric space $X_\infty$, 
\be
X_j \GHto X_\infty,
\ee
then there is a compact metric space, $Z$, and isometric embeddings $\varphi_j: X_j \to Z$ for $j\in \left\{1,2,...,\infty\right\}$ such that
\be
d_H^Z\left(\varphi_j(X_j),\varphi_\infty(X_\infty)\right) \to 0.
\ee
\end{thm}


\subsection{Flat Convergence}\label{ssec-flatN}

In this subsection we briefly describe integral currents on $\R^N$,  the mass of a current and 
the flat distance between integral currents. For more details see the work of Federer and Fleming where these concepts were introduced \cite{FF}.      

Let $M^n \subset \R^N$ be a compact oriented $n$-submanifold and $\tau$ a unit orienting $n$-vector field on $M$. Denote by $\Omega^n_c(\mathbb{R}^N)$ the set of $n$-forms in $\R^N$ with compact support. The functional $[M]$ given by 
\be
\omega \in \Omega^n_c(\mathbb{R}^N)\mapsto [M] (\omega) =\int_M  \langle\omega, \tau\rangle d\mathcal H^n
\ee
is an $n$-dimensional \textit{current} with weight equal to 1.
This concept can be extended to oriented submanifolds built from countable collections of
Lipschitz functions with pairwise disjoint images, $\varphi_i: A_i\subset \R^n \to \R^N$,  allowing integer weights $a_i\in \mathbb Z$. 
Thus, we say that
\be
T(\omega):=\sum_{i=1}^\infty a_i {\varphi_i}_{\#}\lbrack A_i \rbrack \omega=\sum_{i=1 }^{\infty} a_i 
\lbrack A_i \rbrack (\varphi_i^*\omega)
\ee
is an $n$-dimensional integer rectifiable current, provided the mass of $T$ given as the following
weighted volume 
\be\label{eq-massFF}
\mass (T):=\sum_{i=1}^\infty |a_i| \mathcal H^n( \varphi_i(A_i))
\ee
is finite. 
The boundary of $T$, $\partial T$, is the functional defined by 
$
\omega \in \Omega^{n-1}_c(\mathbb{R}^N)\mapsto T(d\omega).
$
In particular, for a compact oriented submanifold $M^n$ 
\be\label{2equ}
\mass([M])= \vol(M)\,\,\, \&\,\,\,
 \partial \lbrack M \rbrack= \lbrack \partial M \rbrack.
\ee

An integral current
is an integer rectifiable current whose boundary
is also an integer rectifiable current.  The space of $n$-dimensional integral currents on $\R^N$ is denoted by $\intcurr_n(\R^N)$ and it includes the ${\bf{0}}$ current given by ${\bf{0}}(\omega)=0$ for any $n$-form $\omega$ in $\R^N$.

Given $T_1, T_2 \in \intcurr_n(\R^N)$ and an open subset $O\subset \R^N$,  the flat distance between $T_1$ and $T_2$ in $O$  is defined to be 
\be
	d^O_{F}(T_1, T_2)= \inf\left\{ \mass(A) + \mass(B) : T_1-T_2 = A + \partial B \mbox{ \rm{in} } O\right\}
\ee
where the infimum is taken over all $A\in \intcurr_n(O)$
and all $B\in \intcurr_{n+1}(O)$.  

Federer and Fleming's compactness theorem states that any sequence of currents $T_i  \in \intcurr_n(\R^N)$ such that
$\mass(T_i)\le V_0$, $\mass(\partial T_i)\le A_0$,
and $\spt T_i \subset K$ for $K$ compact,  has a subsequence that 
converges in flat sense to an integral current of the same dimension, possibly the $\bf{0}$ current.
We recall that the mass is lower semicontinuous with respect to this distance. 


 \subsection{Ambrosio--Kirchheim Integral Currents}\label{ssec-AK}

In \cite{AK} Ambrosio and Kirchheim extended the notion of  integral currents on $\R^N$ to integral currents on a complete metric space $Z$.   In this case, an $n$-dimensional current $T$ acts on
$(n+1)$-tuples of Lipschitz functions $(f, \pi_1,...,\pi_n)$
rather than differential forms and one requires the existence of a finite Borel measure on $Z$, $\mu$, such that 
$$|T(f,\pi_1,...,\pi_n)| \leq   \Lip(\pi_1)\cdots \Lip(\pi_n) \int_Z |f|d\mu,$$
the smallest of such measures is denoted by $||T||$, $\mass(T)=||T||(Z)$ and the set of $T$ is defined as follows:
\begin{equation}
\set(T)= \{ x \in Z \,|  \, \liminf_{r \to 0}  \tfrac{  ||T||(B_r(x))} {\omega_n r^n} > 0 \}
\end{equation}
where $\omega_n$ is the Hausdorff measure of a unit radius ball in the $n$-dimensional Euclidean space.

An $n$-dimensional integer rectifiable current can be written in the following way. 
There exist a countable collection of bi-Lipschitz charts with pairwise disjoint images,
$\varphi_i: A_i \to U \subset Z$ for Borel sets $A_i$
in $\mathbb{R}^n$ and, $a_i \in \mathbb N$, such that
\be
T(f, \pi_1,..., \pi_n)=\sum_{i=1}^\infty a_i \varphi_{i\#} [A_i]
(f, \pi_1,..., \pi_n)
\ee
where the push forward is defined as
\be
{\varphi_i}_{\#}\lbrack A_i \rbrack (f, \pi_1, \dots, \pi_n) 
= \int_{A_i} f\circ \varphi_i \,\,
d(\pi_1\circ\varphi_i) \wedge \cdots \wedge d(\pi_n\circ\varphi_i).
\ee
The  boundary of $T$ is the functional defined as
\be
\partial T(f, \pi_1,...,\pi_n) = T(1, f, \pi_1,...,\pi_n).
\ee
The space of $n$-dimensional integral currents on $Z$, denoted $\intcurr_n(Z)$, is
the collection of $n$-dimensional integer rectifiable currents whose boundaries 
are integer rectifiable.  Again there is the $\bf{0}$ integral current
in each dimension.  

We remark that the mass of a multiplicity one rectifiable current in a metric space is not necessarily equal to its Hausdorff measure, c.f. (\ref{eq-massFF}), though the following relationship holds
between the mass measure and the Hausdorff measure for $n$-dimensional integer rectifiable currents with weights $a_i=1$:
\be
C_n\mathcal{H}^n(\set (T)) \le \mass(T) \le C'_n  \mathcal{H}^n(\set (T)).
\ee
Here, $C_n, C'_n$ are precise constants depending only on the dimension. See  Lemma 9.2 and Theorem 9.5 in \cite{AK}. For $n$-dimensional oriented Riemannian manifolds with finite volume, (\ref{2equ}) still holds in the framework of Ambrosio--Kirchheim.

The notion of flat distance extends to $\intcurr_n(Z)$
which we denote as $d_F^Z(T_1, T_2)$.  Federer and Fleming's compactness theorem also extends to this setting replacing $O$ by a compact metric space $Z$. The mass is lower semicontinuous with respect to $d_F^Z$ given that  flat convergence implies weak convergence and the mass is lower semicontinuous with respect to weak convergence, see the line after Definition 3.6 in \cite{AK} and Remark 3.12 and 3.13 in \cite{SorWen2}.


\subsection{Intrinsic Flat Convergence}

In \cite{SorWen2}  Sormani and Wenger applied Ambrosio and Kirchheim's notion of an integral current to  define $n$-dimensional integral current spaces $(X,d,T)$. Here, the pair $(X,d)$ denotes a metric space, $T \in \intcurr_n(\bar{X})$ and $X=\set(T)$, where $\bar{X}$ denotes the metric completion of $X$. 

The condition $X=\set(T)$ implies that for any collection of bi-Lipschitz charts 
$\varphi_i: A_i \subset \R^n\to \subset X$ 
such that $T=\sum_{i=1}^\infty a_i \varphi_{i\#} [A_i]$ satisfies
\be
\mathcal{H}^n\left(X \smallsetminus \bigcup_{i=1}^\infty \varphi_i(A_i)\right)=0.
\ee
 So integral current spaces are countably $\mathcal{H}^n$-rectifiable
metric spaces endowed with oriented charts and integer weights.
The zero integral current space, $\bf{0}$, in every dimension has current
structure $0$ and no metric space.  

Given an $n$-dimensional integral current space $(X,d,T)$, the triple $\bdry (X,d,T) := (\set(\partial T), d, \partial T \rstr \overline{\set (\bdry T)}) \in  \intcurr_{n-1}(\overline{\set (\bdry T)})$ is an $(n-1)$-dimensional integral current space.  Note that 
$\set(\partial T) \subset \bar X$ and that $\bdry (X,d,T)$ is endowed with the restricted metric from $\bar X$.  

Compact oriented Riemannian manifolds with boundary $(M^n, g_M)$ can be regarded as integral
current spaces, 
$(M,d_M,[M])$, where $d_M$ is the intrinsic Riemannian distance
function. That is, $d_M(x,y)$ is the infimum over the lengths of  continuous curves lying in $M$
joining $x$ to $y$.  The integral current structure $[M]$ is defined by
\be
[M](f, \pi_1,...,\pi_n)=\int_M f\, d\pi_1\wedge\cdots \wedge d\pi_n
\ee
and in this case, $\set([M])=M$ and $ \set (\partial [M])=\partial M$.
 
The definition of intrinsic flat distance is as follows.
\begin{definition}[\cite{SorWen2}]\label{IF-defn} 
Given two $n$-dimensional precompact integral current spaces 
$(X_1, d_1, T_1)$ and $(X_2, d_2,T_2)$, the 
intrinsic flat 
distance between them is defined by
\be
d_{\mathcal{F}}\left( (X_1, d_1, T_1), (X_2, d_2, T_2)\right)=\inf\left\{d_F^Z(\varphi_{1\#}T_1, \varphi_{2\#}T_2):\,\,\varphi_j: X_j \to Z \right\}
\ee
where the infimum is taken over all  complete metric
spaces $Z$ and all metric isometric embeddings.
\end{definition}

We remark that the flat distance is an extrinsic notion. Meanwhile, the intrinsic flat distance is an intrinsic notion. See Example 2.8 in \cite{HLS}.   The function $d_{\mathcal F}$ is a distance on precompact integral current spaces
in the sense that $d_{\mathcal{F}}\left((X_1, d_1, T_1), (X_2, d_2, T_2)\right)=0$ if and only if
there is a current preserving isometry between these spaces.  That is, 
\begin{equation}\label{eq-equalCurrSp}
(X_1, d_1, T_1)= (X_2, d_2, T_2)
\end{equation}
as integral current spaces if and only if there exists 
$\varphi: X_1 \to X_2$ metric isometry such that $\varphi_\# T_1=T_2$.
If $M_i$ are Riemannian manifolds and $(M_1, d_{M_1}, [M_1])= (M_2, d_{M_2}, [M_2])$, this means  
there is an orientation preserving isometry between $M_1$ and $M_2$.

For integral current spaces the following compactness theorem holds. 

\begin{thm}[Wenger~\cite{Wenger-compactness}]\label{thm-WengerComp}  
Let $V_0, A_0, D_0 > 0$ and let $(X_j, d_j, T_j)$
be a sequence of integral current spaces of the same dimension
such that
\be \label{eqn-compact-1}
\diam(X_j) \le D_0,\,\,\,\mass(T_j)\le V_0\,\,\,\&\,\, \mass(\partial T_j) \le A_0.
\ee 
Then there exist a subsequence of $(X_j, d_j, T_j)$ (still denoted $(X_j, d_j, T_j)$) and an
integral current space, $(X_\infty, d_\infty, T_\infty)$,
of the same dimension, possibly the $\bf{0}$ space, such that
\be
\lim_{j\to\infty}d_{\mathcal{F}}\left((X_j, d_j, T_j), (X_\infty, d_\infty, T_\infty) \right)=0.
\ee
\end{thm}

Now we cite the following theorem by Allen, Perales and Sormani that we will use to prove our intrinsic flat result \cite{APS}. 

\begin{thm}[Allen--Perales--Sormani, Theorem 2.1 in \cite{APS}]\label{thm:aps}
Let $M$ be a closed  and oriented smooth manifold and let $g$ and $g_i$, $i\in \mathbb N$, be Riemannian metrics on $M$ such that for all $i$, 
$\diam(M, g_i)  \leq  D$, $g(v,v)\leq g_i(v,v)$ for all $v \in TM$ and,  $\lim_{i \to \infty}\vol_{g_i}(M)=\vol_g(M)$.
Then 
\begin{equation*}
\lim_{i \to \infty} d_{\mathcal F}((M, g_i), (M,g)) =0.
\end{equation*}
\end{thm}

\begin{lemma}[Sormani--Wenger, Lemma 5.4  in \cite{SorWen2}]\label{thm:biLiptoIFdis}
Let $(X, d_X, T)$ be an integral current space, assume that $(X,d_X)$ is a complete metric space. 
If $(Y, d_Y)$ is also a complete metric space and $\phi : X \to Y$ is a $\lambda$-biLipschitz map for some $\lambda > 1$,
then 
\begin{align*}
&  d_{\mathcal F} ( (X,d_X,T),  (  \set(\phi_\sharp T), d_Y, \phi_\sharp T)   \\
& \leq c(\lambda, n) \max \{\diam(X), \diam(\set (\phi_\sharp T) ) \} (\mass(T) + \mass(\bdry T)), 
\end{align*}
where  $c(\lambda,n)= \tfrac{1}{2} (n+1)\lambda^{n-1}(\lambda-1)$.
\end{lemma}}


\section{Graph Tori}\label{sec-GraphTori}

In this section we define the notion of a graph torus $M$ arising from a $C^4$ function $f: T \to \R$, where $T$ is a $3$-dimensional flat torus. We then define the negative scalar curvature excess of $f$, $\m(f) = -\int_T R^-(M) d\vol_T$. Here, $R(M)$ denotes the scalar curvature and \linebreak[4] $R^-(M)= \min\{R(M), 0\}$.   The quantity $\m(f)$ could be seen as the analogue of the ADM mass in \cite{HL}. Note that a small negative lower bound on the scalar curvature of $M$, $R(M)$, translates into a small upper bound for $\m(f)$. Hence, studying the stability of graph tori with almost nonnegative scalar curvature can be done considering $\m(f)$.  

Then we see that $\m(f)$ bounds a weighted total mean curvature integral,
\bee
\m( f) \geq \int_{\S_h}    \frac{|D f|^2}{1+|D f|^2}H_{\S_h} \, d\vol_{\S_h},
\eee
where $h$ is any regular value of $f$, $\S_h$ the corresponding level set and $H_{\S_h}$ its mean curvature.  We are assuming that the standard round sphere in $\R^n$ has positive mean curvature with respect to the inner pointing normal vector. The previous bound will be used to prove a differential inequality for the volume function of the level sets in section \ref{sec-proofsF}.  In subsection \ref{ssec-Ftori} we cover information about fundamental domains and in the last subsection we describe in detail the class of functions $\G$ and the classes of manifolds $\Md$ and $\Mda$ for which we will prove flat and intrinsic flat convergence. 

\subsection{Graph Tori}\label{ssec-graphtori}
We define a class of $3$-dimensional Riemannian manifolds $(M,g_M)$ as follows.

\begin{definition}
We say that $(M,g_M)$ \textit{arises from} $(T,f)$ if  $(T,g_T)$  is a $3$-dimensional flat torus, $f\in C^4(T)$ and
$$M=T \ \ \& \ \ g_M=g_T+ df\otimes df.$$
We call $(M,g_M)$ a \textit{graph torus}.
\end{definition}

\noindent
Consider $T\times \mathbb{R}$ equipped with the flat metric $g_{T\times \mathbb{R}}=g_T+ dt\otimes dt$, and let 
$i: \graph(f)\rightarrow T\times \mathbb{R}$ be the inclusion map, $i(x,t)=(x,t)$. 
We can consider $\graph(f)\subset T\times \mathbb{R}$ equipped with the metric
\begin{align*}
g_{\graph(f)}= i^{*}g_{T\times \mathbb{R}}
\end{align*}
and define the diffeomorphism
\begin{align*}
F: M\rightarrow \graph(f), \ \ \ F(x)=(x,f(x)).
\end{align*}
One can check that $F^*g_{\graph(f)}= g_M$, and therefore we consider in some cases that $M$ is a submanifold of $T\times \mathbb{R}$.

\subsection{The Negative Scalar Curvature Excess}\label{ssec-scalar}

Here we define the negative scalar curvature excess $\m(f)$. In Remark  \ref{rmrk-m/R} we see that if $R(f)$ is almost nonnegative
then $\m(f)$ is bounded above by a small quantity.  Thus we will study the stability of graph tori with almost nonnegative scalar curvature through $\m(f)$.

A direct computation shows that the scalar curvature of $(M,g_M)$ is given by
\begin{align*}
R(M):= R(f):=\dv_{T} \left[ \frac{1}{\sqrt{1+|Df|^2}}(f_{ii}f_j - f_{ij}f_i) \pr_j \right].
\end{align*}
where $(M,g)$ arises from $(T,f)$ and $\dv_{T}$ is the divergence of $(T,g_T)$, c.f. {\cite[Lemma 14]{Lam}. 

By the divergence theorem it follows that
 \be \label{eq-div-0}
 \int_{T} R( f) \, d \vol_{g_T} = 0.
 \ee

\begin{definition} \label{def-m}
Let  $(M,g_M)$ be a graph torus that arises from $(T,f)$. We define the negative scalar curvature excess $\m(f)$ as the nonnegative number given by
\be
\m(f) := - \int_{T} R^-(f)\,d\vol_{g_T} \geq 0,
\ee
where $R^-(f)(x)=\min\{R(f)(x),0\}$. 
\end{definition}
From (\ref{eq-div-0}) it follows that
 \be\label{eq-R+}
 \int_T R^+( f) \, d \vol_{g_T}= \m( f) =-\int_T R^-( f) \, d \vol_{g_T}.
 \ee

\begin{remark}\label{rmrk-m/R}
A negative lower bound on $R(f)$ translates into an upper bound for $\m(f)$. 
More precisely, if $R(f) \geq -\varepsilon$ for some $\varepsilon \geq 0$, then
\item  \be\label{eq-mbound}
0 \leq \m(f) \leq \vol_{g_T}(T) \veps.
\ee
We also note that
\begin{align*}
\left\|R^{\pm}(f)\right\|_{L^1(T,\vol_{g_M})}\geq 
\m(f).
\end{align*}
\end{remark}
 
\begin{remark}
If either $R(f) \geq 0$ or  $R(f)\leq 0$ then $R(f)=0$. 
This follows from $(\ref{eq-R+})$ and the torus rigidity result. 
\end{remark}

We set $\Omega_h= f^{-1}(-\infty,h)$ and $\S_h= \partial \Omega_h$. If $h$ is a regular value of $f$, then $\Sigma_h=\{ f = h \}$. 

\begin{lemma}
For any regular value $h$ of $f$,
\be\label{eq-mest}
\m( f) \geq \int_{\S_h}    \frac{|D f|^2}{1+|D f|^2}H_{\S_h} \, d\vol_{\S_h}.
\ee
\end{lemma}

In the next section we will use this inequality to obtain a lower bound on the derivative of the volume function of the level sets of regular values in terms of the mean curvature, 
$H_{\S_h}$, and $\m(f)$. 

\begin{proof}
Recall  (\ref{eq-div-0}), then by the divergence theorem, 
\begin{align*}
\int_{T \setminus \Omega_h} R( f)\, d \vol_{g_T}
&= - \int_{\S_h} \frac{1}{1+|D f|^2}( f_{ii} f_j -  f_{ij} f_i)\eta^j \, d\vol_{\S_h} , 
\end{align*}
where  $\eta$ is the unit outward normal vector to $\S_h$. Since $\S_h$ is a level set of $f$ for a regular value $h$, if $\eta=\frac{D f}{|D f|}$, we get
\bee
H_{\S_h} = -\dv_{T}\( -\frac{D f}{|D f|} \) = \frac{1}{|D f|^3} ( f_{ii} f_j f_j -  f_i f_j f_{ij}),
\eee
On the other hand, 
\bee
(f_{ii} f_j -  f_{ij} f_i)\eta=\frac{1}{|D f|}( f_{ii} f_j f_j -  f_{ij} f_i f_j)=|D f|^2 H_{\S_h}.
\eee
Hence,
\be \label{eq-intscal}
\int_{T \setminus \Omega_h} R( f)\, d \vol_{g_T} =- \int_{\S_h} \frac{|D f|^2}{1+|D f|^2}H_{\S_h}\, d\vol_{\S_h}.
\ee
Similarly, if $\eta=-\frac{D f}{|D f|}$ we get the same expression (cf. \cite{Lam,HW1}).
From \eqref{eq-div-0} and \eqref{eq-intscal} it also follows that
\bee
\int_{\Omega_h} R( f)\, d \vol_{g_T}= \int_{\S_h} \frac{|D f|^2}{1+|D f|^2}H_{\S_h} \, d\vol_{\S_h}.
\eee
Since
\bee
\m( f) = \int_T R^+( f) \, d \vol_{g_T}  \geq \int_{\Omega_h} R^+( f) \, d \vol_{g_T} \geq  \int_{\Omega_h} R( f) \, d \vol_{g_T}
\eee
we therefore have
\bee
\m( f) \geq \int_{\S_h}    \frac{|D f|^2}{1+|D f|^2}H_{\S_h} \, d\vol_{\S_h}.
\eee
\end{proof}

\subsection{Flat Tori and Fundamental Domains}\label{ssec-Ftori}

Let $(T,g_T)$  be a flat torus. Then there is a subgroup $(\Gamma,+)$ of $(\mathbb{R}^3,+)$ isomorphic to $(\mathbb{Z}^3,+)$ that acts by translations, hence, isometrically, on $(\mathbb{R}^3,g_{eucl})$ such that 
 $(T,g_T)$ is isometric to $(\mathbb{R}^3, g_{eucl})/\Gamma$.

Let $p: \mathbb{R}^3\rightarrow T$ be the quotient map associated to $\Gamma$.  Then $p^*g_T= g_{eucl}$ and $p: (\mathbb{R}^3,g_{eucl})\rightarrow (T,g_T)$ is a local Riemannian isometry.  Therefore, up to isometry,  $p$ is the universal Riemannian covering map of $(T,g_{T})$. The distance function on $T$ induced by $g_T$ is given by
$$
d_{T}(p(z),p(w))=\inf_{\gamma,\beta\in\Gamma} d_{eucl}(\gamma(z),\beta(w))=\inf_{\beta\in\Gamma} d_{eucl}(z,\beta(w)).
$$

Let $\left\{a^1,a^2,a^3\right\}\subset \mathbb{R}^3$ be a set of generators of $\Gamma$ such that $A=(a^1,a^2,a^3)$ is positive definite with respect to the standard orientation of $\mathbb{R}^3$.  Define, 
\begin{align}\label{eq-D}
D =\{ \sum_{i=1}^3 \lambda_i a^i\,|\, \lambda_i\in [0,1] \}
\end{align}
and call it a fundamental domain of $\Gamma$.  Another way to define a fundamental domain is the following,  
$$Q=\{ \sum_{i=1}^3 \lambda_i a^i\,|\, \lambda_i\in [0,1)\} \subset D.$$
In this case, $p|_Q: Q \to T$ is bijective.  For every $z\in \mathbb{R}^3$, we define a fundamental domain that contains $z$ by
$$Q(z):=z+Q=\{z+q : q \in Q\}.$$
We have $Q=Q(0)$.  It is easy to check that for $v=z+\frac{1}{2}\sum_{i=1}^3a_i$ and $w\in Q(z)$ we have
\begin{align}\label{useful1}
d_{T}(p(z),p(w))= \inf_{\beta\in\Gamma} d_{eucl}(z,\beta(w))=d_{eucl}(z,w).
\end{align}
Moreover, setting $z=-\frac{1}{2}(a_1+a_2+a_3)$ we define
\begin{align*}
2Q(z):= \bigcup_{v\in Q(z)}Q(v)=\left\{v+w:v,w\in Q(z)\right\}.
\end{align*}

\begin{lemma}\label{thm-Bbilip}
Let $(T_i,g_{T_i})_{i\in \mathbb{N}}$ be a family of flat Riemannian tori such that $\diam(T_i)\leq D_0$ and $\vol(T_i) \geq \nu$ for some $\nu>0$. 
Then, there is a subsequence and a flat Riemannian torus $(T_{\infty},g_{T_{\infty}})$ such that 
\bee
(T_i,d_{T_i}) \Lipto   (T_{\infty}, d_{T_\infty})
\eee 
and
\be\label{eq-Dlipconv}
(D_{i},d_{eucl}) \Lipto   (D_{\infty}, d_{eucl}).
\ee 
More precisely,  to prove (\ref{eq-Dlipconv}) for the linear transformations 
$$A_i= (a^1_i,a^2_i,a^3_i) :  \R^3  \to \R^3$$
given by the generators of $\Gamma_i$, $i \in \mathbb \cup \{\infty \}$, 
we show that there exist common subsequences denoted in the same way such that 
$$a^1_i \to  a^1_\infty, \quad  a^2_i \to  a^2_\infty \quad a^3_i \to  a^3_\infty,$$ so that
\begin{equation}
A_\infty A^{-1}_i |_{D_i}: D_i   \to D_\infty,
\end{equation}
where $A_\infty= (a^1_\infty,a^2_\infty,a^3_\infty)$,
are bi-Lipschitz with $$\Lip(A_\infty A^{-1}_i), \Lip( (A_\infty A^{-1}_i)^{-1}) \to 1.$$ 

Moreover, from (\ref{eq-Dlipconv}) it follows that
\bee
d_H^{\R^3}((D_i,d_{eucl}), (D_{\infty}, d_{eucl})) \to 0.
\eee 
\end{lemma}

Recall that
if there is a uniform lower bound for the injectivity radius of $T_i$, $i\in \mathbb N$, then there exists $\nu>0$ such that $\vol(T_i)\geq \nu$.

\begin{proof} 
Let $\{a^1_i,a^2_i,a^3_i\} \subset \R^3$ be the set of generators of $\Gamma_i$. Since $\diam(T_i)\leq D_0$ for all $i\in \mathbb{N}$, it follows that 
the euclidean norms of $a^1_i,a^2_i$ and $a^3_i$ are uniformly bounded from above. Hence, we can pick a subsequences that converge to $a_{\infty}^1, a_{\infty}^2$ and 
$a_{\infty}^3$ respectively. 
The lower bound for the injectivity radius implies that $(T_i,g_{T_i})$ is non-collapsing. Hence, there exists $\nu>0$ such that $\det A_i\geq \nu$, and therefore also 
$\det A_\infty \geq \nu$, $A_\infty=(a^1_{\infty},a^2_{\infty},a^3_{\infty})$. In particular $a_{\infty}^1,a_{\infty}^2$ and $a_{\infty}^3$ are linearly independent and $A_\infty$ is positive definite.

The triple
$\{a^1_{\infty},a^2_{\infty},a^3_{\infty}\}$ generates a subgroup $\Gamma_{\infty}$ of $(\mathbb{R}^3,+)$, and $(\mathbb{R}^3,g_{eucl})/\Gamma_{\infty}$ is a flat torus denoted as $(T_{\infty},g_{\infty})$.
Then, it is clear that for the same subsequence $\overline{2Q_i(z_i)}$ with $z_i=-\frac{1}{2}(a^1_i+a^2_i+a^3_i)$ and $D_i$ converges in Lipschitz sense to $\overline{2Q_{\infty}(z_{\infty})}$ and $D_{\infty}$ respectively.

Let us be more precise. After applying a linear transformation without loss of generality we can assume that $\Gamma_{\infty}=\mathbb{Z}^3$. In this case, we have $\overline{2Q_{\infty}(z_{\infty})}=[-1,1]^3$. Then, $A_i=(a^1_i,a^2_i,a^3_i):\mathbb{R}^3\rightarrow \mathbb{R}^3$ is a bi-Lipschitz transformation 
such that $(A_i)_{i\in \mathbb{N}}\rightarrow A_0 $ and
\begin{align}\label{useful2}
\max\left\{ \sup_{v,w\in [-1,1]^3} \frac{|v-w|}{|A_iv-A_iw|},\sup_{v,w\in =[-1,1]^3} \frac{|A_iv-A_iw|}{|v-w|}\right\}\rightarrow 1, \ \ i\rightarrow \infty.
\end{align}
Hence, we obtain Lipschitz convergence of $D_i$.

Let us check Lipschitz convergence of the tori. It is clear that $A_i[v+[0,1)]= Q_i(A_iv)$.
Let $x,y\in T_{\infty}$ and pick $v_x\in Q(-\frac{1}{2}(1,1,1))$ such that $p(v_x)=x$. Then, we define a Lipschitz map between $T_{\infty}$ and $T_i$ via $\phi=p|_{Q_i(A_i(v_x))}\circ A_i \circ (p|_{Q(v_x)})^{-1}$.
We can also pick $w_{x,y}\in Q_{\infty}(v_x)$ with $p(w_{x,y})=y$. Then $\phi(y)=p(A_iw_{x,y})$ with $A_iw_{x,y}\in Q_i(A_iv_x)$. Hence, from (\ref{useful1}) and (\ref{useful2}) we get that
\begin{align*}
\limsup_{i\rightarrow \infty}\!\!\sup_{x,y\in T_{\infty}}\!\!\frac{d_{T_i}(\phi_i(x),\phi_i(y))}{d_{T_{\infty}}(x,y)}&= \limsup_{i\rightarrow\infty}\!\!\sup_{x,y\in T_{\infty}}\!\!\frac{|A_iv_x-A_iw_{x,y}|}{|v_x-w_{x,y}|}\\
&\leq \limsup_{i\rightarrow\infty}\!\!\sup_{v,w\in \overline{2Q(-\frac{1}{2}(1,1,1))}}\!\! \frac{|A_iv-A_iw|}{|v-w|}\leq 1.
\end{align*}
And the same holds after changing the roles of $g_{T_i}$ and $g_{T_\infty}$. Hence, $(T_i,d_{T_i})$ converges in Lipschitz sense to $(T_{\infty},d_{T_{\infty}})$. 
\end{proof}

\subsection{The Classes $\M$ and $\Mda$}\label{ssec-Mda}

Now we define the classes of manifolds we will be working with. First we give the definition of an outer minimizing set in $\R^n$. See \cite{EG}.

For any nonempty set $E \subset  \mathbb R^n$ the perimeter of $E$ is defined as 
$$P(E)=  \sup\{ \int_E  \text{div}(\varphi) d \mathcal L^n \, | \, \varphi \in C_c^1(\mathbb R^n, \mathbb R^n), \,||\varphi||_{\infty} \leq 1 \} \in \mathbb R \cup \{\infty\}. $$
If $P(E) < \infty$, then $P(E)= \mathcal H^{n-1} (\bdry ^* E)$  where $\bdry ^*E$ is the so called reduced boundary, which is an $\mathcal H^{n-1}$ countably rectifiable set.   If   $E \subset \R^n$ is a  bounded set  such that   $\partial E$ is Lipschitz then $\bdry^*E = \bdry E$.
If a sequence of measurable sets $E_j$ and a measurable set $E$ satisfy
$\chi_{E_j}  \to \chi_E$ in $L^1$, then 
\begin{align}\label{ineq:lsc}
P(E) \leq \liminf_{j \to \infty} P(E_j).
\end{align}

\begin{definition}
Let $E \subset \R^n$ be a bounded subset of finite perimeter and $\partial^* E$ be its reduced boundary.  We say that $\partial ^* E$ is outer minimizing if for any bounded set $F \subset \R^n$ containing $E$, 
          $\mathcal H^{n-1}( \partial ^* E) \leq P( F)$. 
\end{definition}

Now we are ready to define $\G$ and $\M$.

\begin{definition}\label{def-famG}
Let $\G$ be the set of functions satisfying:
\begin{enumerate}
\item $f:T^3 \To \R$ for a flat torus $T$
\item $\max_{T} f =(f \circ p) \vert_{\pr D} \equiv 0$, where $p: \R^3 \to T$ is a Riemannian covering map and $D$ the closure of a fundamental domain
\item For almost every $h\in f(T)$ the level set $\Sigma_h= f^{-1}(h)$ is strictly mean convex with respect to $-\frac{Df}{|Df|}$, i.e., ${\bf H}_{\S_h} \cdot -\frac{D f}{|D f|} >0$
\item For almost every $h\in f(T)$ the level set $\tilde \Sigma_h= (f \circ p)^{-1}(h)$ is outer minimizing.
\end{enumerate}
Then let $\mathcal{M}$ be the class of graph tori that arise from all functions in $\G$.
Let $\Mda$ be the class of manifolds $M \in \M$ such that 
$$
\diam (M)\leq D_0 \,\, \&\,\, \minA(M) \geq A_0.
$$
Finally, let $\mathcal M^{D_0}$ be the class of manifolds $M \in \M$  that only satisfy
$\diam (M)\leq D_0$.
\end{definition}

Let us give some brief comments regarding the conditions in Definition \ref{def-famG}. Since we will be following the general strategy of \cite{HL}, we require our functions to have a special ``shape''. In \cite{HL} this desired ``shape'' was implied by assuming the graphs to have nonnegative scalar curvature. In our case, since we are interested in having some negative curvature, we impose $\max_{T} f = 0$ on the elements of $\G$ in condition (2), see Lemma \ref{lem-Vnonin} below. In addition, the fact that the projection of $f$ is constant on $\pr D$ guarantees that we certainly have a function on $T^3$, while the choice of the constant 0 is by simplicity; we could allow the members of $\G$ to have arbitrary constants and perform a normalization by translation as in \cite{HL}, in order to have $\max_{T} f = 0$, to obtain the convergence result. Conditions (2) and (3) will allow us to apply the Minkowski inequality in Lemma \ref{lem-diffineq} to obtain the differential inequality which will be key to obtain the volume estimates in Section \ref{sec-proofsF}.

\begin{lemma}\label{lem-DiamEst}
Let $(M,g_M)  \in \M$ arising from $(T,f)$, then 
\begin{enumerate}
\item $\m(f)=0$ if and only if $f=0$\smallskip
\item $|\min f| \leq \diam(M)$\smallskip
\item $\diam(T) \leq \diam(M)$\smallskip
\item $\diam(D) \leq 2 \diam(T)$\smallskip
\item $\vol(T)=\vol(D) \leq (2 \diam(T))^3$\smallskip
\item $\vol(\bdry D) \leq 6(2 \diam(T))^2$\smallskip\end{enumerate}
If $M \in \Mda$, then
\begin{enumerate}\addtocounter{enumi}{6}\smallskip
\item $\inj(T) \geq \frac{A_0}{2D_0}$  \smallskip
\end{enumerate}
\end{lemma}

\begin{proof}
$(1)$   If $f$ is not constant and $\m(f)=0$ then there is $h$ such that the right hand side of \eqref{eq-mest} is positive. This contradicts $\m(f)=0$. Hence, $f$ is a constant function. By definition of $\G$, $f=0$ on $p(\partial D)$. Hence, $f$ is the zero function.  (2) Observe that the function $\Psi: (M, d_M)  \rightarrow (T  \times \mathbb{R}, d_{T \times \R})$  given by $\Psi(x)= (x,f(x))$ is $1$-Lipschitz.  Since $M \in \M$ there is $y \in M$ such that $f(y)=0$, thus  for any $x\in M$, 
\begin{align*}
|f(x)| \leq d_{T \times \R}((x,f(x)),(y,f(y)) \leq d_M(x,y)\leq \diam (M).
\end{align*}
This implies $|\min f| \leq  \diam (M)$. 

$(3) - (6)$ are easy to check. 

For $(7)$ note that the injectivity radius of $(T,g_T)$ is given by $\inj(T)= \min |a^i|/2$ where $a^1, a^2, a^3$ are the generators of $\Gamma$. Now observe that $A_{ij}=p^{-1}(\{ ta^i+sa^j: t,s\in [0,1]\})$, $i \neq j\in \{1,2,3\}$, is a closed surface in $M$. Since $f(A_{ij})=0$,  $g_M|_{A_{ij}}=g_T|_{A_{ij}}$.  Therefore, 
$$\area(A_{ij})= | a^i \times a^j|=|a^i||a^j|\sin \alpha_{ij}$$ 
where $\alpha_{ij}$ is the angle between $a^i$ and $a^j$. Now $A_{ij}$ is a minimal surface. To see this recall the formula 
\begin{align*}
H^0_{\Sigma}=\frac{ H_{\Sigma}}{\sqrt{1+|\nabla f|^2}}
\end{align*}
for a smooth surface $\Sigma$ contained in a level set $f^{-1}(\{c\})$ of $f$, c.f. \cite{Lam}.
$H^0_{\Sigma}$ is the mean curvature of $\Sigma$ in $T\times \{c\}$ and $H_{\Sigma}$ is the mean curvature with respect to $g_T+df\otimes df$. Since $M\in \mathcal{M}^{D_0}_{A_0}$, 
\begin{align*}
|a^i||a^j|\sin \alpha_{ij}=\area(A_{ij})\geq \min A(M)\geq A_0.
\end{align*}
Hence $\sin \alpha_{ij}|a^i|\geq\frac{A_0}{D_0}$, $i=1,2,3$. Therefore $\inj(T)=\min_{i,j} \frac{\sin \alpha_{ij}|a^i|}{2} \geq \frac{A_0}{2D_0}$.
\end{proof}


\section{Defining and Bounding $\bf{h_0}$, Volume Estimates and Flat Convergence Results}\label{sec-proofsF}

In this section we adapt Huang and Lee's calculations to prove the stability of the positive mass theorem for graphical hypersurfaces of Euclidean space with respect to flat distance to our case. We also adapt Huang, Lee and Sormani's volume estimates that show that the volume is continuous with respect to flat and intrinsic flat distance within the class of graphical hypersurfaces they consider.  Recall that in general the volume function is lower semicontinuous with respect to flat and intrinsic flat distance.

In Subsection \ref{ssec-VolEst1} we divide any graph tori $M \in \M$ in two regions by choosing a suitable value $h_0 \in [\min f, \max f]$ such that 
the first region, $\Omega_{h_0}=f^{-1}(-\infty, h_0)$, has regular level sets with small volume with respect to $\m(f)$. This, together with a control on $h_0 - \min f$ 
provides a volume bound for $\Omega_{h_0}$.   In the second region,  $M \setminus \Omega_{h_0}=f^{-1}[h_0, \max f)$, we show that $\max f - h_0$ can be bounded above by $C \mathcal H^2(\bdry D)^{5/4} \m(f)$.  Here $D \subset \R^3$ denotes a fundamental domain on the universal cover of $T$.  Then  $M \setminus \Omega_{h_0}$  has volume no bigger than the volume of $T$ plus a small quantity that depends on the previous bound.  

Applying the previous results,  we estimate the volume of $M$ in Subsection \ref{ssec-volEst2} and estimate $d_F(M,T \times \{h_0\})$ in terms of the volume of $T$, $\m(f)$ and $\mathcal H^2(\partial D)$ in Subsection \ref{ssec-proofsF}.  At the end we prove flat convergence of sequences $M_i \in T \times \R^3$ with $M_i \in \M$ arising from $(T,f_i)$ provided $\m(f_i) \to 0$.  We obtain Theorem \ref{thm-Flat2} as corollary. 

\subsection{Defining and Bounding $\bf{h_0}$}\label{ssec-VolEst1}
Given a graph torus $M \in \M$ arising from $(T,f)$ we want to divide it in two regions. The first, $\Omega_{h_0}=f^{-1}(-\infty, h_0)$, with level sets of small volume expressed in terms of $\m(f)$.  The second, $M \setminus \Omega_{h_0}$,  with an upper bound on its height, $\max f - h_0$, in terms of $\m(f)$ as well.  For that purpose, in this subsection we study the function $\V(h)= P (f^{-1}(-\infty, h))$, where $P$ denotes the perimeter function. We show that it is non decreasing,  bounded and satisfies a differential inequality in terms of $\m(f)$.  The differential inequality is defined almost everywhere and only at values on which $\V$ is not too small with respect to $\m(f)$.  This constraint comes from the Minkowski inequality and (\ref{eq-mest}).  Taking $h_0$ with $\V(h_0)$ big enough we are able to compare $\V$ with the solutions of a differential equation and derive an upper bound for $|h_0|= \max f - h_0$ in terms of the negative scalar curvature excess.  
\bigskip

Let $(M,g_M)$ be a graph torus arising from $(T,f)$, $p: \R^3 \to T$ be the universal Riemannian covering map
and $D \subset \R^3$ be the fundamental domain of $T$ as in Subsection \ref{ssec-Ftori}.
We define the lift $\tilde f: D \to \R$ of $f$ by setting $\tilde f=f \circ p$ where $p$ is considered to be restricted to $D$.

Since $p^*g_{T}= g_{eucl}$, we have
\be
 \m(\tilde f) : = - \int_{D_T} R^{-}(\tilde f)\,d \vol_{g_{eucl}}= \m( f)
\ee
where 
\begin{align*}
R(\tilde f)=\dv_{\R^3} \left[ \frac{1}{\sqrt{1+|D\tilde f|^2}}(\tilde f_{ii}\tilde f_j - \tilde f_{ij}\tilde f_i) \pr_j \right].
\end{align*} 
Also
\be
\m(\tilde f) \geq \int_{\tilde \S_h}    \frac{|D \tilde f|^2}{1+|D \tilde f|^2}H_{\tilde \S_h} \, d\vol_{\tilde \S_h},
\ee
where $\tilde \Sigma_h=\tilde f^{-1}(h)$ for any regular value $h$ of $\tilde f$.

Since $p^*g_{T}= g_{eucl}$ we can identify $\tilde \Omega_h$ with $\Omega_h$ and $\tilde{\S}_h$ with $\S_h$, and we will do so
in this section. 

We define $ \V:  \tilde f(D)=f(T)\to \R \cup \{\infty\}$  by  
\begin{align} \V(h)= P(\Omega_h)=P(\tilde{\Omega}_h).
\end{align}

\begin{lemma}\label{lem-Vnonin}
Assume that $\tilde{f} \leq 0$,
\begin{align}\label{eq-maxf}
\tilde f|_{\partial D}=0
\end{align}
and that $\tilde{\Sigma}_h\subset \mathbb{R}^3$ is outer minimizing for almost every $h\in f(D)$, then $\V$ is left lower semicontinuous, non-decreasing, $\V(h)=\mathcal{H}^{2}(\S_h)$  for every regular value $h$ of $\tilde f$, and it satisfies $\V \leq \mathcal H^{2} (\partial D)$.
\end{lemma}

\begin{proof}
$\V$ is left lower semicontinuous.  
Taking a sequence $h_j \uparrow h$, we see that $\bigcup_j \Omega_{h_j}=\Omega_h$ and so $\chi_{\Omega_{h_j}}  \to \chi_{ \Omega_h}$ in $L^1$. 
Thus, by (\ref{ineq:lsc}) it follows
$$\V(h) \leq \liminf_{j \to \infty} \V(h_j).$$

By Sard's theorem, the set of critical values of $f$ (or equivalently $\tilde f$) has zero measure, and for every regular value $h$ the level set $\S_h$ is  a 2 dimensional smooth manifold.  Thus, $\V(h)=P(\Omega_h)= \mathcal{H}^{2}(\S_h)$. 

If  $h < 0$  is a regular value of $f$, then $\S_h$ is a compact submanifold of $D^{\circ}$.
Hence, $\V(h)= \mathcal H^{2}  (\S_h)< \infty$. If $\tilde \S_h$  is outer minimizing then $\V(h) \leq P(D)= \mathcal H^{2} (\partial D)$. 
Now we show that $\V(h)   \leq \mathcal H^{2} (\partial D)$ for all $h \in f(T)$ as follows. 
By {left lower semicontinuity, for any $h > \min f$ there is $\varepsilon > 0$ such that if $t \in (h-\varepsilon,h)$} then
$\V(h) \leq \V(t) + \varepsilon$ if $\V(h) < \infty$ or $\V(t) \to \infty$ otherwise.  
By density of regular values we get $\V(h) \leq  \mathcal H^{2} (\partial D)$. 

Now, if $ \min f < h_1 < h_2$, then by the left lower semicontinuity of $\V$ there is {a regular value $h<h_1$} of $\tilde f$ such that $\V(h_1) \leq \V(h) + \varepsilon$. By 
density of regular points and the
outer minimizing assumption, we can assume that $\V(h) \leq \V(h_2)$.
Hence, 
\bee
\V(h_1) \leq \V(h) + \varepsilon \leq \V(h_2) + \varepsilon.
\eee
Taking $\varepsilon$ to zero gives  $\V(h_1) \leq \V(h_2)$. 
\end{proof}

\begin{lemma}\label{lemma-vol-estimate1} 
Let $\tilde f:D \to \R$ be a lifted function with $\tilde f|_{\partial D}=0$ such that  
almost every level set is outer minimizing. Then
for any $\a > 0$ and any regular value $h$ of $\tilde f$ such that $\S_h$ is strictly mean convex with respect to $-\frac{D\tilde f}{|D\tilde f|}$, we have
\be  \label{eq-Vder1}
\V'(h) > \frac{1}{\a} \left(\int_{\S_h} H_{\S_h} \, d \vol_{\S_h}  - (1+\a^{-2})\m(\tilde f)  \right).
\ee
\end{lemma}

\begin{proof}
Applying  (\ref{eq-mest}), for any $\a \geq 0$ we get
\begin{align}\label{eq-malpha}
\m(\tilde f) &\geq  \int_{\S_h} H_{\S_h} \frac{|D\tilde f|^2}{1+|D\tilde f|^2}\, d\vol_{\S_h} \\
&\geq \frac{\a^2}{\a^2+1} \int_{\S_h \cap \{ |D\tilde f| \geq \a \}}H_{\S_h} \, d\vol_{\S_h}. \nonumber
\end{align}

Since $h$ is a regular value of $\tilde f$, $|D\tilde f| \neq 0$ in a neighborhood of $\S_h$. Moreover, $\V$ is finite from Lemma \ref{lem-Vnonin}.
Hence, the first variation formula for the volume of the level sets of $\tilde f$ applies, 
\bee
\tfrac{d}{dt}|_{t=h} \int_{\S_t} d\vol_{\S_t} =  -  \int_{\S_h}  {\bf{H}}_{\S_h} \cdot \frac{D\tilde f}{|D\tilde f|}.
\eee
By hypothesis, ${\bf H_{\S_h}} \cdot -\frac{D\tilde f}{|D\tilde f|} >0$. As in \cite[Lemma 3.4]{HL} we get, 
\begin{align*}
\V'(h) &= \int_{\S_h} \frac{H_{\S_h}}{|D\tilde f|} \, d\vol_{\S_h} \\
&= \int_{\S_h  \cap \{ |D\tilde f| \geq \a \}} \frac{H_{\S_h}}{|D\tilde f|} \, d \vol_{\S_h} +  \int_{\S_h  \cap \{ |D\tilde f| < \a \}} \frac{H_{\S_h}}{|D\tilde f|} \, d\vol_{\S_h}  \\
&> \frac{1}{\a} \int_{\S_h  \cap \{ |D\tilde f| < \a \}} H_{\S_h} \, d\vol_{\S_h}   \\
&=\frac{1}{\a} \left(\int_{\S_h} H_{\S_h} \, d \vol_{\S_h}  -   \int_{\S_h  \cap \{ |D\tilde f| \geq \a \}} H_{\S_h} \, d\vol_{\S_h}  \right) \\
&\geq \frac{1}{\a} \left(\int_{\S_h} H_{\S_h} \, d\vol_{\S_h}  - (1+\a^{-2})\m(\tilde f)  \right),
\end{align*} 
where in the last inequality we used (\ref{eq-malpha}). 
\end{proof}

\begin{lemma} \label{lem-diffineq}
Let $\tilde f:D \to \R$ be a lifted  function   such that  
almost every level set is outer minimizing with $\tilde f|_{\partial D}=0$ and $\m(\tilde f) > 0$. Suppose $h$ is a regular value of $\tilde f$ such that $\S_h$ is strictly mean convex with respect to $-\frac{D\tilde f}{|D\tilde f|}$ and it is outer minimizing and, $\V(h) > \frac{\m(f)^2}{16 \pi}$.  Then,
\be \label{eq-diffineq}
\V'(h) >  \frac{2}{3 \sqrt{3}} \m(\tilde f) \left[ \frac{4 \sqrt{\pi} \V(h)^{1/2}}{\m(\tilde f)} - 1 \right]^{3/2}.
\ee
\end{lemma}

\begin{proof}
Since $\S_h$ is strictly mean convex and outer minimizing, the Minkowski inequality holds,  see \cite[Theorem 2 (b)]{FS}. Thus, we have
$$\int_{\S_h} H_{\S_h}  d\vol_{\S_h} \geq C_n \V(h)^{\tfrac{n-2}{n-1}}.$$
Inserting the Minkowski inequality into \eqref{eq-Vder1} with $n=3$ yields, $C_n= 4 \sqrt{\pi}$ and 
\be \label{eq-Vder2}
\V'(h) > \frac{1}{\a} \left( 4\sqrt{\pi} \V(h)^{1/2} - (1+\a^{-2})\m(\tilde f)  \right),
\ee
for any $\a > 0$.

Now we consider the right hand side of \eqref{eq-Vder2} as a function depending on $\a \in (0,\infty)$.
We see that a global maximum is attained at
\bee
\a = \left[  \frac{3 \m(\tilde f)}{4 \sqrt{\pi} \V(h)^{1/2} - \m(\tilde f)}  \right]^{1/2}. 
\eee
Note that $\a$ is well defined since  $\V(h) > \frac{\m(\tilde f)^2}{16 \pi}$.
Plugging $\a$ into (\ref{eq-Vder2}), we obtain the desired estimate,
\bee
\V'(h) > \frac{2}{3 \sqrt{3}} \m(\tilde f) \left[ \frac{4 \sqrt{\pi} \V(h)^{1/2}}{\m(\tilde f)} - 1 \right]^{3/2}. 
\eee
\end{proof}

\begin{definition}\label{def-h0}
Let $\xi > 0$.  For any function $\tilde f:D \to \R$ such that for almost every value $h$ of $\tilde{f}$ the level set $\tilde \S_h$ is outer minimizing we define
$$A_{\xi}:= \left\{h: \mathcal V(h)\leq (1+\xi)^2\frac{\m(\tilde f)^2}{16\pi}\right\}$$ 
and
\bee
h_0(\xi) = 
\begin{cases}
 \sup \{ h \, | \, 
 h\in A_{\xi}
 \}  \quad&  \text{if $A_{\xi}\neq \emptyset$}  \medskip \\
\min \tilde f     \quad&  \text{otherwise.} \\
\end{cases}
\eee
\end{definition}

We fix $\xi \geq 1$ in the previous definition and write $h_0$ instead of $h_0(\xi)$.
It follows that for $\tilde f$ non constant either
\begin{enumerate}
\item  $h_0  \in (\min \tilde f, \max  \tilde f]$ and thus 
\begin{align*}
\tilde \Omega_{h_0}\neq \emptyset \ \ \ \&\ \ \ D\setminus \tilde \Omega_{h_0} \neq \emptyset
\end{align*}
\item  $h_0= \min \tilde f$ and thus  
$$\tilde \Omega_{h_0}=\emptyset  \ \ \ \&\ \ \ D\setminus \tilde \Omega_{h_0} \neq \emptyset,$$ 
\end{enumerate}
where $\tilde \Omega_{h_0}= \tilde f^{-1}(-\infty, 0)$.

Similarly to \cite[Theorem 3.10]{HL}, we prove that when $h_0 \neq \min \tilde f$ then $|h_0|=\max \tilde f - h_0$ can be bounded in terms of $\m(\tilde f)$. First, we state an ODE comparison lemma. 

\begin{lemma}[{\cite[Lemma 3.9]{HL}}]\label{lem-3.9HL}
Let $V: [a,b] \to \R$ be a non decreasing function and $F:  [a,b] \to \R$ a continuously differentiable and non decreasing function such that 
$V' \geq F(V)$ almost everywhere in $[a,b]$. Let $Y:[a,b] \to \R$ be a $C^2$ function satisfying $Y'=F(Y)$ and $Y(a) \leq V(a)$. 
Then $Y \leq V$ on $[a,b]$. 
\end{lemma}

\begin{lemma} \label{lemma-h0-estimate}
Let $f:T \to \R$ be a smooth function with $\max_T f=\tilde f_{\partial D}=0$. Suppose that for almost every $h \in \tilde f(D)$, the level set $\tilde \S_h$ is strictly mean convex and outer minimizing.  
If  $\m(\tilde f) > 0$, then
\bee
|h_0| < C (\mathcal{H}^2(\partial D))^{1/4} \m(\tilde f)^{1/2}.
\eee
\end{lemma}

\begin{proof}
If $h_0=\max f =0$ then there is nothing to prove.  Assume $h_0 < 0$.  Consider the differential equation
\be\label{eq-Y}
\begin{cases}
Y'(h)  = \frac{2}{3 \sqrt{3}} \m(\tilde f) \left[ \frac{4 \sqrt{\pi} Y(h)^{1/2}}{\m(\tilde f)} - 1 \right]^{3/2}  \medskip \\
Y(h_0 ) = (1+\xi)^2 \frac{\m(\tilde f)^2}{16\pi}
\end{cases}
\ee
This equation has a $C^2$ solution. 
Now from the definition of $h_0$,  for any regular value $h \geq h_0 $ of $\tilde f$,  $\V(h) \geq (1+\xi)^2 \frac{\m(\tilde f)^2}{16\pi}  >   \frac{\m(\tilde f)^2}{16\pi} $.  Hence, the differential inequality of Lemma \ref{lem-diffineq} holds for almost every $h \geq h_0 $. Thus, the hypotheses of Lemma \ref{lem-3.9HL} are satisfied. Then for all $h \geq h_0 $ it holds,
$$Y(h) \leq \V(h).$$ 

From the previous inequality we obtain an upper estimate for $|h_0 |$. 
Using separation of variables and integrating from $h_0 $ to $h$ we get
\bee
\frac{\m(\tilde f)(4\sqrt{\pi} \sqrt{Y} - 2\m(\tilde f))}{4\pi \sqrt{\frac{4 \sqrt{\pi} \sqrt{Y}}{\m(\tilde f)}-1}} \bigg|_{Y(h_0 )=  (1+\xi)^2 \frac{\m(\tilde f)^2}{16\pi} }^{Y(h)} = \frac{2}{3 \sqrt{3}} \m(\tilde f) (h -h_0 ).
\eee
Evaluating the previous expression we get
\bee
\frac{\sqrt{\m(\tilde f)}(4\sqrt{\pi} \sqrt{Y(h)} - 2\m(\tilde f))}{4\pi \sqrt{ 4\sqrt{\pi} \sqrt{Y(h)}-\m(\tilde f)}}   -\frac{ (\xi-1)\m(\tilde f)}{4\pi \sqrt{\xi}} = \frac{2}{3 \sqrt{3}} (h -h_0 ).
\eee

We now obtain an upper bound for the right hand side of the previous inequality.  First, 
since $\m(\tilde f)>0$ and $\xi\geq 1$, we can get rid of the second term on the left hand side. Second, 
since $\m(\tilde f)>0$,  get  $ 4\sqrt{\pi} \sqrt{Y(h)} - \m(\tilde f) > 4\sqrt{\pi} \sqrt{Y(h)} - 2\m(\tilde f)=  (4\sqrt{\pi} \sqrt{Y(h)} - \m(\tilde f))- \m(\tilde f)$. Hence, 
\bee
\frac{\sqrt{\m(\tilde f)}}{4\pi }  \sqrt{ 4\sqrt{\pi} \sqrt{Y(h)}-\m(\tilde f)}   \geq \frac{2}{3 \sqrt{3}} (h -h_0 ).
\eee
Once more, since $\m(\tilde f)>0$,
\bee
\frac{\sqrt{\m(\tilde f)}}{4\pi }  \sqrt{ 4\sqrt{\pi} \sqrt{Y(h)}}   \geq \frac{2}{3 \sqrt{3}} (h -h_0 ).
\eee
Recalling that $Y(h) \leq \V(h)$ for all $h \geq h_0 $, 
\bee
\frac{\sqrt{\m(\tilde f)}}{4\pi }  \sqrt{ 4\sqrt{\pi} \sqrt{\V(h)}}   \geq \frac{2}{3 \sqrt{3}} (h -h_0 ).
\eee
From Lemma \ref{lem-Vnonin} we know that $\V(h)\leq \mathcal{H}^2(\partial D)$. It follows, 
\bee
\frac{\sqrt{\m(\tilde f)}}{4\pi }   \sqrt{ 4\sqrt{\pi} \sqrt{\mathcal{H}^2(\partial D)}}   \geq \frac{2}{3 \sqrt{3}}|h_0 |.
\eee
Thus, 
\be
|h_0| < C   \sqrt[4]{\mathcal{H}^2(\partial D)} \sqrt{\m(\tilde f)},
\ee
with $C=\frac{3\sqrt{3}}{8\pi}\sqrt{4\sqrt{\pi}}= \frac{3\sqrt{3}}{4\pi}\sqrt{\sqrt{\pi} }  $.
\end{proof}

\begin{remark}
We can prove a statement similar to Lemma \ref{lemma-h0-estimate} for $n >3$.  Indeed,
note that Lemma \ref{lemma-vol-estimate1} holds independently of having chosen $n=3$.
Then in Lemma \ref{lem-diffineq} we would have to use the Minkowski inequality for the appropriate dimension $n >3$.
This would lead us to an ordinary differential equation similar to (\ref{eq-Y}).  Finally, the solution to this new equation together with analogous 
calculations to the ones we performed in the proof of Lemma \ref{lemma-h0-estimate} appear in \cite[Lemma 3.4]{ACS}.  Hence, for $n>3$ we also have 
$|h_0| \leq c(\m(f),  \mathcal H^n(\bdry D),n)$ such that $c(\m(f), \mathcal H^n(\bdry D),n) \to 0$ as $\m(f)  \to 0$.  Promoting our results to higher dimensions.
\end{remark}

\subsection{Volume Estimates}\label{ssec-volEst2}

In this subsection we estimate $\vol_{g_M}(\Omega_{h_0})$ and $\vol_{g_M}(M \setminus \Omega_{h_0})$ for graph tori $M \in \M$.  Here, $\Omega_{h_0}= f^{-1}(-\infty, h_0)$ and $h_0$ is as in Definition \ref{def-h0}.  Hence, we obtain an upper bound for the volume of $M$.  From the definition of $g_M$ we also get a lower bound. See Corollary \ref{cor:limsup}.  From this estimates it will follow that the volume function is continuous with respect to flat and intrinsic flat convergence for the theorems stated in Section \ref{sec-intro}.

\bigskip

From the previous subsection if $f \in \G$ either 
\begin{itemize}
\item $\m(f)=0$ and then $M=T$. Hence, $\vol(M)=\vol(T)$, or
\item  $\m(f) \neq 0$ and there exists $|h_0| \leq C  \mathcal H^2(\bdry D)^{5/4}\m(f)^{1/2}$ that either
\begin{enumerate}
\item  $h_0  \in (\min f, \max f]$ and thus 
\begin{align*}
\Omega_{h_0}, M\setminus \Omega_{h_0} \neq \emptyset,\\
\V(h) \leq (1+\xi)^2 \tfrac{\m(f)^2}{16\pi}  \,\,\text{for}\,\, h \leq h_0, \\
\end{align*}
\item  $h_0= \min f$ and thus  
$$\Omega_{h_0}=\emptyset, \,\,\,  M\setminus \Omega_{h_0} \neq \emptyset.$$ 
\end{enumerate}
\end{itemize}

\begin{lemma}\label{lem-volOmega}
Let $(M,g_M) \in \M$ arise from $(T,f)$ with $\m(f) >0$.
Then
\begin{align*}
\vol_{g_M}(\Omega_{h_0})\leq \frac{(1+\xi)^{3}}{6\sqrt \pi} \frac{\m(f)^{3}}{(16\pi)^{\frac{3}{2}}}+  (1+\xi)^2\frac{\m(f)^2}{16\pi}
 |\min f|.
\end{align*}
Here $h_0$ is given as in Definition \ref{def-h0}. 
\end{lemma}

\begin{proof} 
Let $h_1<h_0$  be a regular value of $f$. The co-area formula yields
\begin{align*}
\vol_{g_M}(\Omega_{h_1})=  &  \int_{\Omega_{h_1}} \sqrt{1+|\nabla f|^2} d\vol_{g_T}\\
\leq & \vol_{g_T}(\Omega_{h_1})+ \int_{\min f}^{h_1}\mathcal{V}(h)dh.
\end{align*}
 By the
the isoperimetric inequality, the definition of $h_0$ (Definition \ref{def-h0}) and $\max f=0$, 
\begin{align*}
\vol_{g_M}(\Omega_{h_1})\leq &  \frac{1}{6\sqrt\pi} \mathcal{V}(h_1)^{\frac{3}{2}}+ (1+\xi)^2\frac{\m(f)^2}{16\pi} |\min f|\\
 \leq  & \frac{1}{6\sqrt\pi} \mathcal{V}(h_1)^{\frac{3}{2}}+  (1+\xi)^2 \frac{\m(f)^2}{16\pi} |\min f| \\
\leq  & \frac{(1+\xi)^{3}}{6\sqrt \pi} \frac{\m(f)^{3}}{(16\pi)^{\frac{3}{2}}}+ (1+\xi)^2\frac{\m(f)^2}{16\pi} |\min f|.
\end{align*}
By continuity of $h\mapsto \vol_{g_M}(\Omega_h)$ the claim follows.
\end{proof}

\begin{lemma}\label{lem-volOmegaC}
Let $(M,g_M)  \in \M$ arise from $(T,f)$ with $\m(f) >0$.
Then
\begin{align*}
\vol({M}\backslash \Omega_{h_0})\leq  \vol(T) + C \mathcal{H}^2(\partial D) ^{5/4}\m(f)^{\frac{1}{2}}.
\end{align*}
Here $C$ is the constant that appears in Lemma \ref{lemma-h0-estimate}.
\end{lemma}

\begin{proof} 
Let $h_1\in (h_0-\eta,h_0)$ be a regular value of $f$ for $\eta>0$ sufficiently small.  Then
\begin{align*}
\vol(M\backslash \Omega_{h_1})&=\int_{T \backslash \Omega_{h_1}} \sqrt{1+|\nabla f|^2} d\vol_{g_T}\\
&\leq \vol_{g_T}(T)+ \int_{h_1}^0 \mathcal{V}(h)dh \\
& \leq \vol_{g_T}(T) + |h_1| \mathcal{H}^2(\partial D),
\end{align*}
where in the last line we used the bound on $\V$ given by Lemma \ref{lem-Vnonin}.  Since $|h_1|\leq |h_0|+\eta$ and $|h_0|\leq C  \mathcal{H}^2(\partial D) ^{1/4}    \m(f)^{\frac{1}{2}}$ by Lemma \ref{lemma-h0-estimate}. It follows
\begin{align*}
\vol_{g_M}( M\backslash \Omega_{h_1})\leq \vol_{g_T}(T) + \left(C  \mathcal{H}^2(\partial D)^{1/4}\m(f)^{\frac{1}{2}} +\eta\right)\mathcal{H}^2(\partial D).
\end{align*}
We use continuity of $h\mapsto \vol_{g_M}(\Omega_h)$ and let $h_1\rightarrow h_0$. 
Hence, the previous estimate holds with $h_1$ replaced by $h_0$. Since $\eta>0$ was arbitrary, the claim follows.
\end{proof}

Putting the two previous lemmas together we obtain as a corollary

\begin{coro}\label{cor:limsup}
Let $(M,g_M)  \in \M$ arising from $(T,f)$. Then
\begin{align*}
\vol(T) \leq  \vol(M)  \leq   &     \frac{(1+\xi)^{3}}{6\sqrt \pi} \frac{\m(f)^{3}}{(16\pi)^{\frac{3}{2}}} +  (1+\xi)^2\frac{\m(f)^2}{16\pi}|\min f|\\
& + \vol(T) + C \mathcal H^{2}(\bdry D)^{5/4} \m(f)^{\frac{1}{2}}.
\end{align*}
\end{coro}

\begin{proof}
If $\m(f)=0$ then by Lemma \ref{lem-DiamEst}, $f=0$ and thus there is nothing to prove. Otherwise, 
adding the inequalities coming from Lemma \ref{lem-volOmega} and Lemma \ref{lem-volOmegaC},
 \begin{align*}
\vol(M)  \leq   &     \frac{(1+\xi)^{3}}{6\sqrt \pi} \frac{\m(f)^{3}}{(16\pi)^{\frac{3}{2}}} + (1+\xi)^2\frac{\m(f)^2}{16\pi} |\min f|\\
& +\vol(T) +  C \mathcal H^2(\bdry D)^{5/4} \m(f)^{\frac{1}{2}}.
\end{align*}
Finally, since $g_M=g_T + df^2$,  
\bee
\vol(T)   \leq \int_{T} \sqrt{1+|\nabla f |^2} d\vol_{g_T} = \vol_{g_M}(M).
\eee
This concludes the proof. 
\end{proof}

\subsection{Flat Convergence: Theorem \ref{thm-Flat2}}\label{ssec-proofsF}
In this subsection we prove flat convergence of sequences $M_i \in T \times \R$ with $M_i \in \M$ arising from $(T,f_i)$
provided $\m(f_i) \to 0$.  We get as a corollary Theorem \ref{thm-Flat2}. 
\bigskip

Throughout this subsection  we denote integral currents of the form $[S]$ as $S$. In this case, we will have that $\mass(S)=\vol(S)$.  Recall that 
from subsection \ref{ssec-VolEst1} for any $ M \in \M$  arising from $(T,f)$ with $\m(f) \neq 0$, either
\begin{enumerate}
\item  $h_0  \in (\min f, \max f]$ and thus 
\begin{align*}
\Omega_{h_0}, M\setminus \Omega_{h_0} \neq \emptyset
\end{align*}
\item  $h_0= \min f$ and thus  
$$\Omega_{h_0}=\emptyset, \,\,\,  M\setminus \Omega_{h_0} \neq \emptyset,$$ 
\end{enumerate}
where $\Omega_{h_0}= f^{-1}(-\infty, 0)$.

\begin{thm}\label{thm-Flat11m}
Let $(M,g_M)$ be a graph torus in $\M$ that arises from $(T,f)$.  Then there exist $h_0 \in \R$, $C,c > 0$ such that for all $L \in ( |h_0| , \infty]$
\begin{align*}
d^{T \times (-L, \infty)}_F( M,  T  \times \{h_0\})  \leq & C \mathcal H^2(\bdry D)^{1/4} \vol(D) \m(f)^{1/2} + \\
&  c ( (1+\xi)^2/16\pi)^{3/2} \min\{|\min f|, L\}  \m(f)^3. 
\end{align*}
where $T  \times \{h_0\}$ has the orientation induced by the orientation of $T \times \R$.

Furthermore, 
\begin{align*}
| \vol(M) - \vol(T) |  \leq   & \frac{(1+\xi)^{3}}{6 \sqrt \pi} \frac{\m(f)^{3}}{(16\pi)^{\frac{3}{2}}} + (1+\xi)^2\frac{\m(f)^2}{16\pi} |\min f|+\\
&  C \mathcal H^2(\bdry D)^{5/4} \m(f)^{\frac{1}{2}}.
\end{align*}
\end{thm}

\begin{proof}
If $\m(f)=0$ then by Lemma \ref{lem-DiamEst} we have that $f=0$. Thus,  $M= T \times \{0\} \subset T \times \R$ and $h_0=0$ by definition.  Hence, the theorem holds. 

Assume that $\m(f)> 0$. We will choose an integral current $B$ such that
$M -   T \times \{h_0\}= \pr B$.  Thus by definition of flat distance,  
\bee
d^{T \times (-L, \infty)}_F(M,  T  \times \{h_0\})  \leq \mass(B \cap T \times (-L, \infty)).
\eee

Let $h_{0}$ be the number given in Definition \ref{def-h0}
corresponding to the graph torus $M$. If $h_0 \in (\min f, \max f)$ we proceed in the following way.
Define
$$B_+= \{  (x,t) \in T \times \R \, | \, x \in T,\,\, h_0 \leq t \leq f(x) \}$$ 
and 
$$B_- =  \{  (x,t) \in T \times \R \, | \, x \in  T,\,\, f(x) \leq t \leq h_0 \}.$$
Let the integral current given by $B_+$ have positive orientation
and the one given by $B_-$ have negative orientation.
Set $B=B_+ + B_-$. Note that,
$$M - T  \times \{h_0 \}=  \pr B.$$
Then
\bee
(M - T  \times \{h_0\} )  \cap T \times (-L, \infty)    =  \pr B \cap T \times (-L, \infty). 
\eee
and
\bee
d^{T \times (-L, \infty)}_F(M,  T  \times \{h_0\})  \leq \mass(B \cap T \times (-L, \infty)).
\eee

Since $B_+$ and $B_-$ are disjoint we can calculate $\bM(B_+   \cap T \times (-L, \infty) )$ and $\bM(B_-   \cap T \times (-L, \infty) )$ separately.  

By Lemma \ref{lemma-h0-estimate}, $|h_0| \leq C  \mathcal H^2(\bdry D)^{1/4} \m(f)^{1/2}$. 
Since $L > |h_0|$, it follows that $B_+  \cap T \times (-L, \infty) = B_+$. 
Hence, 
 \begin{align}\label{eqB+}
\vol(B_+  \cap T \times (-L, \infty)) = \vol(B_+)  \leq   & \int_{h_0}^{\max f =0} \int_{T} d\vol_{g_T}dt \\
  =   & \int_{h_0}^{\max f =0} \int_{\D} d\vol_{g_{eucl}}dt  \nonumber \\
 \leq  & C  \mathcal H^2(\bdry D)^{1/4} \vol(D) \m(f)^{1/2}. \nonumber
 \end{align}
 \bigskip

To estimate $\vol(B_-  \cap T \times (-L, \infty))$ write 
\be\label{eq-Bvol}
 \vol(B_-    \cap T \times (-L, \infty)) = \int_{\max\{\min f, -L\}}^{h_0} \vol(B_-\cap  T \times \{h\} )\, d\vol_{g_T}.
 \ee
Now $\partial B_-\cap  T \times \{h\} = f^{-1}(h)$.  Then at any regular value $h$ of $f$,
by the isoperimetric inequality and the definition of $\V$,  
$$ \vol(B_-\cap  T \times \{h\} )  \leq  c \V(h)^{3/2}.$$
From the definition of $h_0$, we know that for $h \leq h_0$
$$ \V(h) < \V(h_0)= (1 + \xi)^2 \frac{\m(f)^2}{16\pi}.$$
Therefore,  (\ref{eq-Bvol}) has the upper bound 
 \begin{align}\label{eqB-}
 \vol(B_-  \cap T \times (-L, \infty)) = & \int_{\max\{\min f, -L\}}^{h_0} \vol(B_-\cap  T \times \{h\} )\, d\vol_{g_T} \nonumber
 \\
 \leq & c ( (1+\xi)^2/16\pi)^{3/2}  \min\{|\min f|, L\}  \m(f)^3,
 \end{align}
where we used the fact that $h_0 \leq 0$.  

\bigskip
If  $h_0= \max f$ then 
$$M - T  \times \{h_0 \}=  \pr B_-.$$
If  $h_0= \min f$ then 
$$M - T  \times \{h_0 \}=  \pr B_+.$$ 
In all cases, from (\ref{eqB+}) and (\ref{eqB-}) we get, 
\begin{align*}
d^{T \times (-L, \infty)}_F(M,  T  \times \{h_0\})  \leq  & C \mathcal H^2(\bdry D)^{1/4} \vol(D) \m(f)^{1/2} + \\
&  c ( (1+\xi)^2/16\pi)^{3/2} \min\{|\min f|, L\}  \m(f)^3.
\end{align*}
The volume estimate follows from Corollary \ref{cor:limsup}.
\end{proof}

\begin{thm}\label{thm-Flat22m}
Let $(M_i,g_{M_i})$ be graph tori in $\M$ arising from $(T,f_i)$, $i \in \mathbb N$. 
If $\m(f_i) \to 0$, then there is a subsequence of $M_i$ which we denote in the same way such that 
for all $L >0$
\bee
d^{T \times (-L, \infty)}_F( M_i,  T  \times \{0\}) \to 0
\eee
where $T  \times \{0\}$ has pointing upward orientation.

If $|\min f_i| \m(f_i)^2 \to 0$ then
\bee
d^{T \times \R}_F(M_i,  T  \times \{0\}) \to 0
\eee
and
$$ \vol(M_i)    \to  \vol(T).$$
\end{thm}

\begin{proof}
If $\m(f_i)=0$ then $M_i= T \times \{0\} \subset T \times \R$. Hence, if  there is an infinite number of $i \in \mathbb N$ such that $\m(f_i)=0$
the theorem holds.  Otherwise, assume $\m(f_i) >0$ for all $i$ and let $h_{0i}$ be the value given in Definition \ref{def-h0}
corresponding to the graph torus $M_i$.  By the triangle inequality, 
\begin{align*}
d^{T \times (-L, \infty)}_F( M_i,  T  \times \{0\}) & \leq 
d^{T \times (-L, \infty)}_F( M_i,  T  \times \{h_{0i}\})  \\
&  +  d^{T \times (-L, \infty)}_F( T  \times \{h_{0i}\},  T  \times \{0\}).
\end{align*}
Since $\m(f_i) \to 0$, $ |h_{0i}| \leq C  \mathcal H^2(\bdry D)^{1/4}\m(f_i)^{1/2} \to 0$.  Thus, we can assume that $L > |h_{0i}|$ and apply
Theorem \ref{thm-Flat11m}, 
\begin{align*}
d^{T \times (-L, \infty)}_F( M_i,  T  \times \{h_{0i}\})  \leq &
 C \mathcal H^2(\bdry D)^{1/4} \vol(D) \m(f_i)^{1/2} \\
+ &  c ( (1+\xi)^2/16\pi)^{3/2} \min\{|\min f_i|, L\}  \m(f_i)^3 \\
\to & 0. \nonumber 
\end{align*}
Furthermore, 
\begin{align*}
d^{T \times (-L, \infty)}_F( T  \times \{h_{0i}\},  T  \times \{0\}) \leq & \vol (T \times [h_{0i},  0]  ) \\
\leq   &  |h_{0i}|   \vol(T) \to 0. 
\end{align*}

\bigskip 
If $|\min f_i| \m(f_i)^2 \to 0$ then from Corollary \ref{cor:limsup} we get $\vol(M_i) \to \vol(T)$. 
To calculate $d^{T \times \R}_F( M_i,  T  \times \{0\})$ we have to modify the calculations of the previous paragraph. 
Here $L = \infty$,  so we only have to ensure that 
$\min\{|\min f_i|, L\} \m(f_i)^3 = |\min f_i| \m(f_i)^3  \to 0$. This limit follows from 
$\m(f_i) \to 0$ and  $|\min f_i| \m(f_i)^2 \to 0$. 
\end{proof}

\begin{remark}\label{rmrk-xibigger}
Notice that under the assumptions of Theorem \ref{thm-Flat22m} if we allow $\xi \geq 1$ to increase then the flat distance between the graph tori and the flat torus increases but we still get 
flat convergence. 
\end{remark}

\begin{proof}[Proof of Theorem \ref{thm-Flat2}]
This follows from the fact that $R(f_i)> -\varepsilon$ implies that $\m(f_i) <  \varepsilon \vol(T)$.  See Remark \ref{rmrk-m/R}.  Also $|\min f_i| \m(f_i)^2 \leq D_0^2\m(f_i)^2 \to 0$. Then the conclusion follows from the previous theorem. 
\end{proof}

\section{Intrinsic Flat Result: Theorem \ref{thm-IF2}}

In this section we prove Theorem \ref{thm-IF2} by applying Theorem \ref{thm:aps}. 

\begin{thm}\label{th:mmm}
Let $M_i \in \Md$, $i \in \mathbb N$, be a sequence of tori arising from $(T_i,f_i)$. Assume that 
$$
\m(f_i) \to 0.
$$
Then either $\vol(M_i)  \to 0$ or there is a subsequence $\{i_k\}$ and a flat torus $(T_\infty, g_{T_\infty})$ such that 
\begin{align}\label{equ:assume}
d_{\mathcal{F}}  (( M_{i_k},d_{M_{i_k}},  [M_{i_k}]), (  T_\infty ,d_{T_\infty}, [T_\infty])) \to 0.
\end{align}
Moreover, 
\be\label{eq-massmmm}
\mass( [M_{i_k}]) \to \mass( T_\infty).
\ee
If $M_i \in \Mda$ then $\vol(M_i)  \to 0$ does not occur.
\end{thm}

\begin{proof}
If $M_i \in \Mda$ then by Lemma \ref{lem-DiamEst} we have $\inj(T_i) \geq \tfrac{A_0} {2D_0}$. Since the tori are flat, $\vol(T_i) \geq \omega_3 (\tfrac{A_0} {2D_0})^3$.
So from now on suppose that $\vol(M_i)  \nrightarrow 0$ then there is a subsequence which we denote in the same way such that 
$\vol(M_i) \geq 2\nu$. By Lemma \ref{lem-DiamEst},   Corollary \ref{cor:limsup} and the hypothesis $\m(f_i) \to 0$ we get
  $\vol(T_i) \geq \nu$ for large $i$.  Then by Lemma \ref{thm-Bbilip} there is a subsequence $\{T_{i_k}\}$ that converges to a flat torus $T_\infty$.  
Because of this we have 
$\vol(T_{i_k})  \to \vol(T_\infty)$. Recalling that  $|\min f|  \leq \diam(M_i)\leq D_0$ and Corollary \ref{cor:limsup} once more we find
$$ \lim_{k \to \infty} \vol(M_{i_k})= \vol(T_\infty).$$
Thus, (\ref{eq-massmmm}) follows since $\mass=\vol$.

 Let  $\phi_{i_k}: T_\infty \rightarrow T_{i_k}$ be the diffeomorphisms given in Lemma \ref{thm-Bbilip}. 
Define Riemannian metrics $\bar g_{i_k}$ in $T_\infty$ as follows
\begin{align*}
\bar g_{i_k}=g_{T_\infty} + d(f_{i_k}\circ \phi_{i_k})\otimes d(f_{i_k}\circ \phi_{i_k}). 
\end{align*}
We claim that 
\begin{align*}
\lim_{k \to \infty}d_{\mathcal{F}}((M_{i_k},  g_{M_{i_k}}), (T_\infty, \bar g_{i_k})) =0
\end{align*}
and that 
\begin{align*}
 \lim_{k \to \infty}  d_{\mathcal F}((T_\infty, \bar g_{i_k}), (T_\infty, g_{T_\infty}) ) = 0. 
\end{align*}
By the triangle inequality, these two claims imply (\ref{equ:assume}).

By construction,  $\phi_{i_k}:  (T_\infty, \bar g_{i_k})  \to (M_{i_k} , g_{i_k})$
are biLipschitz maps with Lipschitz constants $\lambda_k$ converging to $1$.  Furthermore, $\phi^{-1}_{j_k \sharp} [M_{j_k}]= [T_\infty, \bar g_{j_k}]$. 
Thus, we can apply Lemma \ref{thm:biLiptoIFdis}:
\begin{align*}
d_{\mathcal F} ( (M_{i_k},d_{i_k},[M_{i_k}]),  ( \set(\phi^{-1}_{i_k \sharp} [M_{i_k}] ), d_{\bar g_{i_k}},  \phi^{-1}_{i_k\sharp} [M_{i_k}] ) \leq   \\
  \leq  c(\lambda_k, n) \max \{\diam(M_{i_k}), \diam(\set ( \phi^{-1}_{i_k \sharp} [M_{i_k}]) ) \} \mass([M_{i_k}]) 
\end{align*}
where  $c(\lambda_k,n)= \tfrac{1}{2} (n+1)\lambda_k^{n-1}(\lambda_k-1)$ 
and $n$ equals the dimension of $M_i$. Since $\diam(M_i) \leq D_0$, $\mass(M_i)$ is uniformly bounded and $\lambda_k$ converges to $1$ then  $$d_{\mathcal F} ( (M_{i_k},d_{i_k},[M_{i_k}]),  ( \set(\phi^{-1}_{i_k \sharp} [M_{i_k}] ), d_{\bar g_{i_k}},  \phi^{-1}_{i_k\sharp} [M_{i_k}] )   \to 0.$$
Thus, the first claim holds.

The second part of the claim follows applying Theorem \ref{thm:aps}.  
Note that by definition of $\bar g_i$, $\bar g_i \geq g_{T_\infty}$. 
Since $\phi_{i_k}:  (T_\infty, \bar g_{i_k})  \to (M_{i_k} , g_{i_k})$
are biLipschitz maps with Lipschitz constants $\lambda_k$ converging to $1$
and all $M_i$ have a uniform upper diameter bound, we can assume that $\diam(T_\infty, \bar g_{i_k})$ is 
uniformly bounded above.  Since $\det( d \phi_{i_k}) \to 1$ and $\vol(M_{i_k}) \to \vol(T_\infty)$,
\begin{equation*}
\vol(T_\infty, \bar g_{i_k})=   \int_{T_\infty}    \sqrt{1 +  |\grad(f_{i_k}  \circ \phi_{i_k})|^2}  d\vol_{g_{T_\infty}} \to \vol(T_\infty).
\end{equation*}
Thus we can apply Theorem \ref{thm:aps}. 
\end{proof}

\begin{proof}[Proof of Theorem \ref{thm-IF2}]
This follows applying Theorem \ref{th:mmm}. For that, we see that $R(f_i)> -1/i$ implies that $\m(f_i) <  \vol(T_i)/i \leq (2D_0)^3/i \to 0$.  See Remark \ref{rmrk-m/R} and Lemma \ref{lem-DiamEst}. 
 \end{proof}

\small{
\bibliographystyle{amsalpha}
\bibliography{graphTori}
}

\end{document}